\documentclass[11pt]{article}
\usepackage{amssymb, amsmath, amsfonts, epsfig, color,amscd,setspace,lineno}

\newtheorem{conjecture}{Conjecture}[section]

\newtheorem{lemma}[conjecture]{Lemma}
\newtheorem{proposition}[conjecture]{Proposition}

\newtheorem{theorem}[conjecture]{Theorem}

\newenvironment{proof}{\noindent {\bf Proof} \hspace{.1cm}}{\hfill ${\bf QED}$ \\ \vspace{.15cm}} 

\title{Prising apart geodesics by length in hyperbolic manifolds}
\author{James W. Anderson}

\begin{document}

\maketitle

\begin{abstract} 
\noindent In this note, we develop a condition on a closed curve on a surface or in a $3$-manifold that implies that the length function associated to the curve on the space of all hyperbolic structures on the surface or in the $3$-manifold (respectively) completely determines the curve.  Specifically, for an orientable surface $S$ of negative Euler characteristic, we extend the known result that simple curves have this property to curves with self-intersection number one (with one exceptional case arising from hyperellipticity that we describe completely).  For a large class of hyperbolizable $3$-manifolds, we show that curves freely homotopic to simple curves on $\partial M$ have this property. \\


\noindent
keywords: length, hyperbolic metric, simple curves, Horowitz tuples
\end{abstract}

\section{Introduction and statement of results}
\label{introduction}

Randol \cite{randol}, building on earlier work of Horowitz \cite{horowitz} on characters of representations of free groups into ${\rm SL}_2 ({\mathbb C})$, makes the remarkable observation that on an orientable surface $S$ of negative Euler characteristic, there exist pairs of distinct (homotopically non-trivial) closed curves having the property that their lengths are equal to one another in each hyperbolic structure on $S$.  In fact, for any $n\ge 2$, there exist $n$-tuples of closed curves on $S$ whose lengths are equal to one another in each hyperbolic structure on $S$.   

As Randol notes, earlier work of Abraham \cite{abraham} demonstrates that such tuples of curves do not exist for general families of metrics, making the existence of such tuples a phenomenon of constant negative curvature metrics.   (It is not known whether the existence of such tuples characterises or largely characterises constant negative curvature metrics among all metrics.)  Since Randol's observation, much effort has been undertaken to characterise in some way such tuples of curves on a surface; however, such efforts are still incomplete.   We note here in particular the work of Ginzburg and Rudnick \cite{ginzburg rudnick}, in which they develop a condition on the exponents of a word $w$ in the free group of rank two which implies that $w$ cannot belong to any such tuple; the work of Leininger, who shows that the straightforward necessary topological condition for two curves to have the same length is not sufficient; and the work of Masters \cite{masters}, who demonstrates the existence of such tuples in $3$-dimensional hyperbolic manifolds.   We refer the interested reader to the survey by the author \cite{anderson survey}  for a discussion of variants of this question and known results.  

In this note, we take a different tack and consider an obverse question to the characterisation question described briefly above.  Specifically, we develop a condition on a closed curve on a surface or in a $3$-manifold that implies that the curve cannot belong to one of the tuples described above, for either surfaces or $3$-manifolds.  That is, we develop a condition that implies that the length function associated to a closed curve on the space of all hyperbolic structures on the surface or in the $3$-manifold (essentially) completely determines the curve.  We also discuss why no stronger similar conditions hold, at least of the sort discussed herein.

We note here that in her thesis, Bright \cite{bright} used different techniques to give partial results towards answering this question of developing such a condition for core curves in a book of $I$-bundles $M$ which implies that there is no other curve in $M$ with the same length over all hyperbolic structures on ${\rm int}(M)$.  

McShane \cite{mcshane homeomorphisms} shows that the length function associated to a simple curve $C$ on an orientable hyperbolic $S$ completely determines $C$.    We extend this to show, for a surface $S$ which is the interior of a compact orientable surface of negative Euler characteristic, that for a curve $C$ on $S$ with self-intersection number one, the length function $\ell_C$ associated to $C$ completely determines $C$, with two non-avoidable exceptions related to hyperellipticity.

\begin{theorem} Let $S$ be the interior of a compact, orientable surface $\Sigma$ of negative Euler characteristic.  Let $C$ be a curve on $S$ with self-intersection number one.  If $C'$ is a curve on $S$ that satisfies $\ell_{C'} (\rho) = \ell_C (\rho)$ for all $\rho\in {\cal D}(\pi_1(S))$, then either $C' = C$ or we are in the case of the hyperelliptic exception, so that one of the following holds:
\begin{enumerate}
\item We have that $S = \Sigma_2$ is the closed, orientable surface of genus two with hyperelliptic involution $\tau$, the curve $C$ is contained in the interior of a pair of pants $P\subset \Sigma_2$ for which every boundary component of $P$ is a non-separating curve on $\Sigma_2$, and $C' = \tau(C)$. 
\item There exists an embedded torus with one hole $T\subset \Sigma$ containing $C$, that $\tau$ is the hyperelliptic involution restricted to $T$, and that $C' = \tau(C)$.  
\end{enumerate}
\label{one-self-intersection}
\end{theorem}

The proof of Theorem \ref{one-self-intersection} follows relatively straightforwardly from standard properties of hyperbolic metrics on surfaces, including the Collar Lemma and the fact that the length of a non-simple curve on a hyperbolic surface has a positive universal lower bound, independent of the surface and the curve, over all hyperbolic structures on the surface.  

The main result of this note is to present the following extension to a wide class of $3$-manifolds with interiors admitting hyperbolic metrics.  (Full definitions are given in Section \ref{preliminaries}.)

\begin{theorem} Let $M$ be a compact, hyperbolizable $3$-manifold with non-empty, incompressible, atoroidal boundary.  Assume that $M$ is not an $I$-bundle over a surface.  Let $C$ be a  curve in $M$ freely homotopic to a simple curve on $\partial M$. If $C'$ is a curve in $M$ which satisfies $\ell_{C'} (x) = \ell_C (x)$ for all hyperbolic structures $x$ on ${\rm int}(M)$, then $C' = C$.  
\label{separating-by-length}
\end{theorem}

We first prove a restricted version of Theorem \ref{separating-by-length} for acylindrical $3$-manifolds $M$ and then use topological arguments to handle the general case.  

We would like to thank the referee for their careful reading of the paper and their comments, and in particular for pointing out that a mistake in the statement and proof in the original version of Theorem \ref{one-self-intersection}, which led to the current phrasing of this Theorem and the Lemmas leading up to this statement, in particular Lemma \ref{punctured torus}.

\section{Curves on surfaces}
\label{curves on surfaces}

We begin by stating Horowitz's original result.  Recall that a {\em primitive element} of a (finitely generated) free group $F$ is an element that belongs to a free basis of $F$, where a {\em free basis} for a (finitely generated) free group is a generating set of minimal cardinality.  We wish to consider all realisations of $F$ as a subgroup of ${\rm SL}_2 ({\mathbb C})$, so we define the representation space
\[ {\cal R}(F) = \{ \rho: F\rightarrow {\rm SL}_2({\mathbb C})\: | \: \rho\mbox{ is a homomorphism} \}, \]
with the natural topology induced by choosing a free basis $\{ f_1,\ldots, f_r\}$ for $F$ and realizing ${\cal R}(F)$ as a subset of $({\rm SL}_2({\mathbb C}))^r$ via the map $\rho \mapsto (\rho(f_1),\ldots, \rho(f_r))$.  

Each non-trivial element $f\in F$ induces a {\em character}, which is the function $\chi[f]: {\cal R}(F)\rightarrow {\mathbb C}$ given by 
\[ \chi[f](\rho) = {\rm tr}(\rho(f)). \]
In his original paper, Horowitz \cite{horowitz} proved the following result for (finitely generated) free groups.

\begin{theorem} [see Theorem 7.1 and Corollary 7.2 of Horowitz \cite{horowitz}] Let $u$ be an element of a free group $F$.  If $\chi[u] = \chi[a^m b^n]$, where $m$ and $n$ are integers (allowing the possibility that either $m$ or $n$ be $0$) and $a$ and $b$ are primitive elements of $F$, then $u$ is conjugate to $(a^m b^n)^{\pm 1}$.
\label{horowitz primitive}
\end{theorem}


It is unclear what is the cleanest algebraic generalisation of Theorem \ref{horowitz primitive} to more general words in a finitely generated free group $F$.   Along these lines, we highlight the work of Ginzburg and Rudnick \cite{ginzburg rudnick}, who develop the following condition.    Let $w = a^{m_1} b^{n_1} \cdots a^{m_p} b^{n_p}$ be any word in the free group $F = {\rm free}(a,b)$ of rank two.   They first observe that the word $w' = I(w) =  b^{n_p} a^{m_p} \cdots b^{n_1} a^{m_1}$ obtained by writing $w$ backwards has the same character as $w$; that is, $\chi[w] = \chi[I(w)]$ for all $w \in F$.  Define the vector ${\mathbb R} = (r_1,\ldots, r_p)$ of non-zero integers to be {\em non-singular} if $r_k \ne \sum_{j\in S} r_j$ for every $1\le k\le p$ and every subset $S\subset \{ 1,\ldots, p\}$, $S\ne \{ k\}$.  They then show that if both the exponent vector $(m_1,\ldots, m_p)$ for the powers of $a$ in $w$ and the exponent vector $(n_1,\ldots, n_p)$ for the powers of $b$ are non-singular, then (up to inverse and conjugacy) the only possible word with the same character as $w$ is $I(w)$.

At this point, we shift our focus to the analogous question for lengths of curves rather than characters.   Let $\Sigma$ be a compact, orientable surface of negative Euler characteristic, possibly with boundary, whose interior $S = {\rm int}(\Sigma)$ then admits a {\em hyperbolic structure}, by which we mean a complete Riemannian metric of curvature $-1$, possibly of infinite area; where relevant, we let $S_x$ denote $S$ equipped with the hyperbolic structure $x$.  

A non-trivial element $c\in \pi_1(S)$ is {\em maximal} if it is not the proper power of another element of $\pi_1(S)$, or equivalently if $\langle c\rangle$ is a maximal cyclic subgroup of $\pi_1(S)$.  (We will participate in the standard abuse of notation and normally repress the base-point when discussing fundamental groups.)  We note that every non-trivial element of $\pi_1(S)$ is either maximal or a proper power of a maximal element, as elements of the fundamental group of a surface are not infinitely divisible. 

A {\em curve} $C$ on $S$ is the free homotopy class corresponding to a maximal element $c\in \pi_1 (S) = \pi_1(\Sigma)$, so that in particular we are explicitly excluding proper powers $c^n$ for $|n|\ge 2$ in this definition.    We work throughout with {\em unoriented curves}, so that for a maximal element $c\in \pi_1(S)$ with its corresponding curve $C$, we have that $C$ is also the curve corresponding to $c^{-1}$.  Since curves are free homotopy classes of elements of $\pi_1(S)$, the curves on $S$ naturally correspond to conjugacy classes of maximal cyclic subgroups of $\pi_1(S)$. 

A curve $C$ on $S$ is {\em simple} if it contains a simple representative, by which we mean a non-self-intersecting, closed loop.  Otherwise, if no such simple representative exists, the curve $C$ is {\em non-simple}.   A curve $C$ on $S\subset \Sigma$ is {\em peripheral} if $C$ is freely homotopic to a component of $\partial \Sigma$.   Note that peripheral curves are necessarily simple.

For a non-simple curve $C$, the {\em self-intersection number} of $C$ is the minimum number of times any representative in the free homotopy class $C$ intersects itself. 

Two curves $C$ and $C'$ are {\em disjoint} if they contain disjoint representatives; otherwise, no such disjoint representatives exist and the curves {\em intersect}.   Note that a peripheral curve on $S$ is necessarily disjoint from every other curve on $S$.

A useful property of hyperbolic structures is that for each non-peripheral curve $C$ on $S$ and each hyperbolic structure $x$ on $S$, there exists a unique closed geodesic $C^\ast$ on $S_x$ in the free homotopy class $C$.  Moreover, the self-intersection number of the free homotopy class $C$ is  realised by the self-intersection number of the geodesic representative $C^\ast$ of $C$. 

The situation for peripheral curves is slightly more complicated.  Namely, there are two possible types of end for $S$ with a given hyperbolic structure $x$.   One type of end is a {\em funnel}, so that there exists a simple, closed geodesic bounding an exponentially flaring end of $S_x$ homeomorphic to an annulus.  Such an end exists if and only if $S_x$ has infinite area.  The {\em convex core} of the surface $S_x$ is the result of cutting $S_x$ along the simple closed geodesics bounding the funnel ends and removing the open funnels.

The other type of end is a {\em cusp}, which is conformally equivalent to a punctured disc.  For a peripheral curve $C$ homotopic into a cusp on $S_x$, there is no closed geodesic in the free homotopy class $C$; instead, there exist a sequence of representatives of $C$ whose lengths on $S_x$ go to zero.   The cusp ends of a hyperbolic surface are contained in the convex core of the surface. 

Phrased in terms of geodesics, two non-peripheral curves $C_0$ and $C_1$ on $S$ intersect if and only if their geodesic representatives $C_0^\ast$ and $C_1^\ast$ intersect for some, and hence every, hyperbolic structure on $S$.  

We now consider the space of {\em marked hyperbolic structures} on $S$.  A {\em Fuchsian group} $\Phi$  is a discrete subgroup of the group ${\rm Isom}^+({\mathbb H}^2) \cong {\rm PSL}_2({\mathbb R})$ of orientation-preserving isometries of the real hyperbolic plane ${\mathbb H}^2$.  Let ${\cal D}(\pi_1 (S))$ be the space of realisations of $\pi_1(S)$ as a Fuchsian group, so that 
\begin{eqnarray*}
{\cal D}(\pi_1 (S)) & =   \{ & \rho: \pi_1 (S)\rightarrow {\rm PSL}_2({\mathbb R})\: |\: \rho(\pi_1 (S))\mbox{ is Fuchsian} \\ 
& & \mbox{with quotient }{\mathbb H}^2/\rho(\pi_1(S))\mbox{ homeomorphic to }S \}. 
\end{eqnarray*}
(The restriction that ${\mathbb H}^2/\rho(\pi_1(S))$ be homeomorphic to $S$ is necessary here as we may be in the case that $S$ is not closed, in which case $\pi_1(S)$ is free and there may be multiple topological types of such surfaces with isomorphic fundamental groups.)  Where appropriate, we let $S_\rho$ be the surface $S$ with the hyperbolic structure coming from the representation $\rho$, so that $S_\rho = {\mathbb H}^2/\rho(\pi_1(S))$. 

Each hyperbolic structure on $S$ arises from a representation  $\rho\in{\cal D}(\pi_1 (S))$, which yields both the hyperbolic structure by taking the quotient $S_\rho = {\mathbb H}^2/\rho(\pi_1(S))$, together with the marking of $\pi_1 (S)$ by $\rho$, which allows us to distinguish between curves.  To each curve $C$ on $S$, we associate the function 
\[ \ell_C: {\cal D}(\pi_1(S))\rightarrow {\mathbb R} \]
given by setting $\ell_C (\rho)$ to be the length of the closed geodesic $C^\ast$ corresponding to the free homotopy class $C$ on $S_\rho = {\mathbb H}^2/\rho(\pi_1(S))$.  

If there is no closed geodesic in the free homotopy class $C$, in the case $C$ is peripheral  on $S$ and homotopic to a cusp of $S_\rho$, or equivalently when $\rho(C)$ is a parabolic cyclic conjugacy class, we set $\ell_C (\rho) = 0$.  By the above discussion, the function $\ell_C$ is well-defined.   

We use the following two important results about the behavior of the lengths of curves on hyperbolic surfaces.  The first is a consequence of the Collar Lemma for hyperbolic surfaces.

\begin{lemma} [see Corollary 4.1.2 of Buser \cite{buser-book}] Let $S$ be the interior of a compact, orientable surface of negative Euler characteristic.  Let $C$ be a simple curve on $S$ and let $C'$ be a curve on $S$ that intersects $C$.  We then have that 
\[ \sinh\left( \frac{1}{2} \ell_C (\rho) \right) \: \sinh\left( \frac{1}{2} \ell_{C'} (\rho) \right) > 1 \]
for every $\rho\in {\cal D}(\pi_1(S))$.
\label{collar lemma}
\end{lemma}

The main use we make of the Collar Lemma is to show that two intersecting curves cannot both be represented by short geodesics  in any hyperbolic structure on the surface.   

The second result concerns the behavior of non-simple curves.  We state the version most relevant to the discussion at hand, though we do note that this result has been considerably extended.   For this, we refer the interested reader in particular to the papers of Basmajian \cite{basmajian-stable}, \cite{basmajian universal}.

\begin{lemma} [see for instance Hempel \cite{hempel lower}]  There exists a constant $K>0$ so that if $S$ is the interior of a compact, orientable surface of negative Euler characteristic and if $C$ is a non-simple curve on $S$, then $\ell_C (\rho)\ge K$ for all $\rho\in {\cal D}(\pi_1(S))$. 
\label{hempel}
\end{lemma}

The following result is due to McShane \cite{mcshane homeomorphisms}, and is in a real sense the starting point for the investigations detailed in this note.  A proof for closed surfaces can also be found as Lemma 6.2 in Anderson \cite{anderson survey}.  For the sake of completeness, and because some of the arguments herein arise again in later arguments, we outline the general proof here.

\begin{theorem} Let $S$ be the interior of a compact, orientable surface of negative Euler characteristic.  Let $C$ be a simple curve on $S$.  If $C'$ is a curve on $S$ that satisfies $\ell_C (\rho) = \ell_{C'} (\rho)$ for all $\rho\in {\cal D}(\pi_1(S))$, then $C' =  C$. 
\label{mcshane surface}
\end{theorem}

\begin{proof} Since $C$ is simple, there exists a sequence $\{ \rho_n\}\subset {\cal D}(\pi_1(S))$ so that $\ell_C (\rho_n)\rightarrow 0$.  One standard way of constructing such a sequence is to first find a pants decomposition of $S$ containing $C$; by a {\em pants decomposition}, we mean a collection $P$ of disjoint, simple curves on $S$ so that each component of the complement of (a collection of disjoint representatives for the curves in) $P$ in $S$ is homeomorphic to the thrice-punctured sphere.  A pants decomposition gives rise to the set of {\em Fenchel-Nielsen coordinates} on the Teichm\"uller space of $S$, given by the lengths of the geodesic representatives of the curves in $P$ and the twists along which these curves are glued together; see Abikoff \cite{abikoff}. The length of $C$ is then one of the coordinates and can take any value in $(0,\infty)$.  

Assume first that $C'$ is non-simple.  By Lemma \ref{hempel}, there exists a constant $K >0$ so that $\ell_{C'}(\rho)\ge K$ for all $\rho\in {\cal D}(\pi_1(S))$.  However, we have assumed that $\ell_C (\rho_n) = \ell_{C'} (\rho_n)$ for all $n$ and we have from the previous paragraph that $\ell_C (\rho_n)\rightarrow 0$ as $n\rightarrow\infty$, which is a contradiction.  Hence, it must be that $C'$ is simple.

Suppose now that $C$ and $C'$ intersect, and recall that we have a sequence $\{ \rho_n\}\subset {\cal D}(\pi_1(S))$ for which $\ell_C (\rho_n)\rightarrow 0$ as $n\rightarrow\infty$.  By Lemma \ref{collar lemma}, we have that 
\[ \sinh\left( \frac{1}{2} \ell_C (\rho_n) \right) \: \sinh\left( \frac{1}{2} \ell_{C'} (\rho_n) \right) > 1\]
for all $n$ and so $\sinh\left( \frac{1}{2} \ell_{C'} (\rho_n)\right)\rightarrow \infty$ as $n\rightarrow \infty$, and hence that $\ell_{C'} (\rho_n) \rightarrow \infty$, again contradicting the assumption that $\ell_C (\rho_n) = \ell_{C'} (\rho_n)$ for all $n$.

Hence, we have  either that $C = C'$ or that $C$ and $C'$ are disjoint.  However, if $C'$ is disjoint from $C$, there exists a pants decomposition $P'$ containing both $C$ and $C'$.  Using Fenchel-Nielsen coordinates for $P'$, there exists a sequence $\{ \rho_n\}\subset {\cal D}(\pi_1(S))$ for which $\ell_C (\rho_n) \rightarrow 0$ and $\ell_{C'} (\rho_n) \rightarrow\infty$, a contradiction.  Therefore, we must have that $C' = C$, as desired.
\end{proof}

We note here that Theorem \ref{mcshane surface} fails at this level of generality if we restrict our attention to only those hyperbolic metrics of finite area on $S$.   Let $\Sigma$ be any compact, orientable surface with negative Euler characteristic and with at least two boundary components, let $S ={\rm int}(\Sigma)$, and let 
\[ {\cal D}_{\rm finite}(\pi_1 (S)) = \{ \rho\in {\cal D}(\pi_1(S))\: |\:  {\mathbb H}^2/\rho(\pi_1(S))\mbox{ has finite area} \}. \]
Let $C$ and $C'$ be  peripheral curves corresponding to distinct boundary components of $\Sigma$, and note that for all $\rho\in {\cal D}_{\rm finite}(\pi_1(S))$, we have that $\rho(C)$ and $\rho(C')$ are parabolic conjugacy classes corresponding to distinct cusps on $S_\rho$.  In particular, we have that $\ell_C (\rho) = \ell_{C'} (\rho) = 0$ for all $\rho\in {\cal D}_{\rm finite}(\pi_1(S))$, even though $C'\ne C$.  

A natural question to ask is the extent to which the condition of simplicity in Theorem \ref{mcshane surface} can be relaxed.   The difficulty with this question is finding an appropriate condition on curves.  A natural notion of complexity to use is the self-intersection number of $C$.   We will show that curves of self-intersection number one are essentially characterised by their length functions, with one unavoidable but completely describable exception arising from hyperellipticity.  

We build up to the general case gradually, beginning with the case that $\Sigma$ is a pair of pants.  Recall that a {\em pair of pants} is the compact surface whose interior is homeomorphic to the thrice-punctured sphere.

\begin{lemma} Let $S$ be the interior of a pair of pants $\Sigma$.  Let $C$ be a closed curve with self-intersection number one on $S$ and let $C'$ be a curve on $S$ for which $\ell_C (\rho) = \ell_{C'} (\rho)$ for all hyperbolic structures $\rho\in {\cal D}(\pi_1(S))$.  Then $C' = C$.
\label{three holed sphere}
\end{lemma}

\begin{proof} Let $C$ be a curve on $S$ with self-intersection number one.   Up to conjugation in ${\rm PSL}_2 ({\mathbb R})$, there is a unique hyperbolic structure $\rho_0\in {\cal D}(\pi_1(S))$ so that all three ends of ${\mathbb H}^2/\rho_0(\pi_1(S))$ are cusps; the existence of such a hyperbolic structure is standard, arising for instance by doubling an ideal hyperbolic triangle, and the uniqueness essentially  follows from the (oriented) triple-transitivity of the action of ${\rm PSL}_2 ({\mathbb R})$ on the boundary at infinity ${\mathbb S}^1_\infty = \overline{{\mathbb R}}$ of ${\mathbb H}^2$.  We call both the hyperbolic structure $\rho_0$ and the quotient surface ${\mathbb H}^2/\rho_0(\pi_1(S))$ the {\em thrice-punctured sphere}.

We now use the refinement of Lemma \ref{hempel}, proven independently by Hempel \cite{hempel lower}, Nakanishi \cite{nakanishi} and Yamada \cite{yamada1}, \cite{yamada2} (see also Rivin \cite{rivin}), that over all hyperbolic structures over all surfaces of all (allowable) topological types, the  shortest curve is the figure-eight curve on the thrice-punctured sphere (of which there are three).  By a {\em figure-eight} curve on $S$, we mean exactly what the reader would expect, namely the curve formed by going around once around one peripheral curve of $S$ and then around a second, distinct peripheral curve of $S$, crossing itself once in the process.

Since $S$ is a planar surface, it is easy to see (independent of the hyperbolic structure on $S$) that any curve with self-intersection number one is in fact one of the figure-eight curves on $S$.  In particular, the curve $C$ is a figure-eight curve.  We now consider the behavior of $C$ in the previously discussed hyperbolic structure $\rho_0$ on $S$.  By the result mentioned above, $\rho_0(C)$ is a  shortest curve in any hyperbolic structure on any surface.  

Since $\ell_C (\rho_0) = \ell_{C'} (\rho_0)$, we must then have that $C'$ is also a shortest curve on the thrice-punctured sphere, and hence a shortest curve over all hyperbolic structures on all surfaces.  Hence, we have that $C'$ is necessarily a figure-eight curve on the thrice-punctured sphere, and in particular $C'$  has self-intersection number one.  

Suppose that $C\ne C'$, so that $C$ and $C'$ are both figure-eight curves made up of different peripheral curves.  Let $\rho\in {\cal D}(\pi_1(S))$ be a hyperbolic structure for which all the ends of ${\mathbb H}^2/\rho(\pi_1(S))$ are funnels.   (We make this choice so that we can vary the lengths of the peripheral geodesics in the argument that follows.)   Lift the representation $\rho$ to a representation $R: \pi_1(S)\rightarrow {\rm SL}_2({\mathbb R})$; by a {\em lift}, we mean that $R$ satisfies the equation $\Pi\circ R = \rho$, after setting $\Pi: {\rm SL}_2({\mathbb R}) \rightarrow {\rm PSL}_2({\mathbb R})$ to be the canonical projection.  In this case, such a lift always exists; we refer the interested reader to Kra \cite{kra lifting} for a proof and for a brief survey of the history of independent solutions to this lifting question.  

Fix a basepoint $z_0\in S$ for $\pi_1(S)$ and let $e_1$, $e_2$, and $e_3$ be oriented simple loops based at $z_0$, disjoint except at $z_0$, so that the $e_j$ go around the ends of $S$.  (In particular, after forgetting orientation, the free homotopy classes of $e_1$, $e_2$, and $e_3$ are representatives of the peripheral curves on $S$.)  We can choose the labels and orientations of the $e_j$ so that $e_3 = e_1\cdot e_2$ and that $C = e_1\cdot e_2^{-1}$, and $C' =  e_1\cdot e_3 = e_1^2 \cdot e_2$.  (Here, we are using $\cdot$ to denote the concatenation of loops, read left to right.)

As discussed in Horowitz \cite{horowitz}, the traces ${\rm tr}(R(e_1))$, ${\rm tr}(R(e_2))$, and ${\rm tr}(R(e_3))$ are independent variables.  Using (many times) the following basic properties of trace for $2\times 2$ matrices $A$ and $B$, that ${\rm tr}(A) = {\rm tr}(A^{-1})$, that trace is invariant under conjugation, and that ${\rm tr}(AB^{-1}) = {\rm tr}(A)\: {\rm tr}(B) -{\rm tr}(AB)$, we can express the trace of any word in $A$ and $B$ as a polynomial in ${\rm tr}(A)$, ${\rm tr}(B)$ and ${\rm tr}(AB)$.  

Consider now the value of ${\rm tr}(R(e_1\cdot e_2^{-1})) - {\rm tr}(R(e_1^2 \cdot e_2))$.  Using the properties of trace noted above, we have that
\[ {\rm tr}(R(e_1\cdot e_2^{-1})) - {\rm tr}(R(e_1^2 \cdot e_2)) = ({\rm tr}(R(e_1))- 1)({\rm tr}(R(e_2)) +{\rm tr}(R(e_3))). \]
In particular, we see that by varying ${\rm tr}(R(e_1))$, ${\rm tr}(R(e_2))$, and ${\rm tr}(R(e_3))$ independently, we can ensure that ${\rm tr}(R(e_1\cdot e_2^{-1})) - {\rm tr}(R(e_1^2 \cdot e_2))$ is non-zero.

There is a subtlety here, in that when we lift, we may not necessarily know the sign of the traces of the lifted elements.  The simple solution is to  run through the same argument as just given for ${\rm tr}(R(e_1\cdot e_2^{-1})) + {\rm tr}(R(e_1^2 \cdot e_2))$, showing that
\[ {\rm tr}(R(e_1\cdot e_2^{-1})) + {\rm tr}(R(e_1^2 \cdot e_2)) = ({\rm tr}(R(e_1))+ 1)({\rm tr}(R(e_2)) +{\rm tr}(R(e_3))). \]
Hence, the same argument applies, regardless of the sign that arises in the lifting of $\rho$ to $R$, and so the same conclusion holds.

A straightforward calculation shows that for a hyperbolic element $\gamma$ of ${\rm PSL}_2({\mathbb R})$, the square of the trace ${\rm tr}^2(\gamma)$ of $\gamma$ and the length $\ell(\gamma)$ of the corresponding geodesic ${\rm axis}(\gamma)/\langle\gamma\rangle$ are related by 
\[ {\rm tr}^2(\gamma) = 4\cosh^2\left( \frac{\ell(\gamma)}{2}\right). \]
In particular, since ${\rm tr}^2 (\rho(e_j)) = {\rm tr}^2 (R(e_j))$, we see that elements with equal traces (up to sign) correspond to geodesics with equal lengths and vice versa.

Since the lengths $\ell_C(\rho)$ and $\ell_{C'} (\rho)$ depend only on ${\rm tr}(\rho(e_1\cdot e_2^{-1}))$ and ${\rm tr}(\rho(e_1^2 \cdot e_2))$, respectively, this is sufficient to guarantee that we can find hyperbolic structures $\rho$ on $S$ so that $\ell_C(\rho)\ne\ell_{C'} (\rho)$.  This completes the proof of the Lemma.
\end{proof}

We make the following observation about the proof of Lemma \ref{three holed sphere}.  The important first step is to show that $C$ being a figure-eight curve then implies that $C'$ is also a figure-eight curve, which we did by considering the hyperbolic structure which gave $C$ its shortest length.  However, it is not necessary to consider the hyperbolic structure on $S$ in which all of the ends are cusps.  

By Basmajian \cite{basmajian universal}, we know that among all hyperbolic structures on all surfaces, there is a realised minimum for the lengths of curves with each self-intersection number greater than one (and that these minima increase as the self-intersection number increases). Hence, it suffices in the proof of Lemma \ref{three holed sphere} to consider only those hyperbolic structures on the interior $S$ of the pair of pants for which the lengths of the figure-eight curves are strictly less than the minima for higher self-intersection numbers.  This observation plays an important role in the proofs of future Lemmas.

We next consider two exceptional cases, which arise from similar considerations that underlie the result of Ginzburg and Rudnick \cite{ginzburg rudnick} noted above.   (Looking ahead, the connection is that the isomorphsim on the free group of rank two defined by sending the two generators to their inverses is essentially the hyperelliptic involution on the torus with a single hole.)

Let $\Sigma$ be a compact orientable surface (possibly with empty boundary) and let $S$ be the interior of $\Sigma$.   We say that a hyperbolic structure $\rho$ on $S$ is {\em hyperelliptic} if there exists an orientation-preserving involution $\tau$ of $S$ which induces an isometric involution $\tau_\rho$ of $S_\rho$.   Note that if $\rho$ is a hyperelliptic hyperbolic structure on $S$, then for every curve $C$ on $S$, the lengths of $C$ and $\tau_\rho(C)$ necessarily satisfy $\ell_C (\rho) = \ell_{\tau_\rho(C)} (\rho)$.  

We extend this definition to say that the surface $\Sigma$ itself is {\em hyperelliptic} if the discussion in the preceeding paragraph holds for every hyperbolic structure on $S = {\rm int}(\Sigma)$, that is, so that there exists an orientation-preserving (topological) involution $\tau$ of $S$ which induces an isometric involution $\tau_\rho$ of $S_\rho$ for every hyperbolic structure $\rho$ on $S$.   Where necessary, we refer to $\tau$ as the {\em hyperelliptic involution} on $\Sigma$.

With this definition, there are two hyperelliptic surfaces, the closed orientable surface $\Sigma_2$ of genus two and the torus with a single hole $\Sigma_{1,1}$.  (For more information on hyperelliptic  surfaces, we refer the reader to any introductory text on Riemann surfaces, such as Farkas and Kra \cite{farkas kra}.)   On $\Sigma_2$, we have a complete understanding of how the hyperelliptic involution behaves with respect to simple curves.

\begin{theorem} [Haas and Susskind \cite{haas susskind}] Let $\tau$ be the hyperelliptic involution on a closed Riemann surface $S$ of genus two.  We then have that $\tau(C) = C$ for every simple curve on $S$.  Moreover, let $\alpha$ be a simple closed geodesic on $S$.  If $\alpha$ is a separating geodesic, then $\tau$ preserves the orientation of $\alpha$, and if $\alpha$ is non-separating, then $\tau$ reverses the orientation of $\alpha$. 
\label{haas susskind}
\end{theorem}

There are two possible types of pants decompositions of the closed, orientable surface $\Sigma_2$ of genus two.  One consists of two pairs of pants, where each boundary component of each pair of pants is a non-separating curve on $\Sigma_2$.  The other consists of two tori each with a single hole, where the boundary component of each torus is a separating curve on $\Sigma_2$ and where the curves of the pants decomposition are the common boundary of the two tori, along with a simple non-separating curve on each torus.  One particular consequence of Theorem \ref{haas susskind} is that in the former case, the hyperelliptic involution $\tau$ on $\Sigma_2$ interchanges the interiors of the two pairs of pants.   In the latter case, the hyperelliptic involution takes each of the tori to itself. 

Combining this observation with Lemma \ref{three holed sphere} yields the following partial result.

\begin{lemma} Let $\Sigma_2$ be the closed, orientable surface of genus two, and let $\tau$ be the hyperelliptic involution on $\Sigma_2$.  Let $C$ be a curve on $\Sigma_2$ with self-intersection number one, and assume that $C$ is contained in the interior of a pair of pants $P\subset \Sigma_2$ for which every boundary component of $P$ is a non-separating curve on $\Sigma_2$.  If $C'$ is any curve on $\Sigma_2$ with $\ell_C(\rho) = \ell_{C'} (\rho)$ for all hyperbolic structures $\rho$ on $\Sigma_2$, then either $C' = C$ or $C' = \tau(C)$. 
\label{genus two}
\end{lemma}

\begin{proof}  Since for this pants decomposition the hyperelliptic involution interchanges the interiors of the two pairs of pants, we have that $C$ and $\tau(C)$ are disjoint.  

We first observe that if $C'$ is any curve on $\Sigma_2$ with $\ell_C(\rho) = \ell_{C'} (\rho)$ for all hyperbolic structures $\rho$ on $\Sigma_2$, then $C'$ is disjoint from $\partial P$.  Indeed, if there is a component $\gamma$ of $\partial P$ which intersects $C'$,  then consider any sequence $\{ \rho_n\}$ of hyperbolic structures on $\Sigma_2$ for which $\ell_{\gamma} (\rho_n)\rightarrow 0$ as $n\rightarrow\infty$ and along which the lengths of the other 2 curves in $\partial P$ remain constant as $n\rightarrow\infty$ (constructed as above in the proof of Lemma \ref{mcshane surface}).   By Lemma \ref{collar lemma}, we then have that $\ell_{C'} (\rho_n)\rightarrow \infty$ as $n\rightarrow\infty$, while $\ell_C (\rho_n)$ remains bounded as $C$ is disjoint from $\gamma$, a contradiction.

Therefore, either $C'\subset P$ or $\tau(C')\subset P$.  In either case, we complete the proof using Lemma \ref{three holed sphere} (and the remark immediately following its proof), together with the observation that (as a consequence of hyperellipticity), there is a one-to-one correspondence between hyperbolic structures on $P$ for which all ends are funnels and hyperbolic structures on $\Sigma_2$.  Indeed, any hyperbolic structure $\rho$ on $P$ for which all ends are funnels can be doubled across the convex core of $P_\rho$ using $\tau_\rho$ to obtain a hyperbolic structure on $\Sigma_2$, and by hyperellipticity, every hyperbolic structure on $\Sigma_2$ arises in this way.   

In the case that $C'\subset P$, we have that $C$ and $C'$ are two curves in the pair of pants $P$, where $C$ has self-intersection number one and $\ell_C (\rho) = \ell_{C'} (\rho)$ for every hyperbolic structure $\rho$ on $P$ for which the ends of $P_\rho$ are all funnels.  By Lemma \ref{three holed sphere}, we then have that $C = C'$.   

In the case that $\tau(C')\subset P$, then the argument just given yields that $C = \tau(C')$, and by applying $\tau$ again, we obtain that $C' = \tau(C)$. 
\end{proof}

The remaining preliminary case before we address the general discussion is that $C$ is contained in a torus with a hole.

\begin{lemma} Let $\Sigma_{1,1}$ be the torus with a single hole with interior $S$ and let $\tau$ be the hyperelliptic involution on $\Sigma_{1,1}$.   Let $C$ be a curve on $S = {\rm int}(\Sigma_{1,1})$ with self-intersection number one.    Then, for any curve $C'\ne C$ on $S$ with self-intersection number one and with $\ell_C (\rho) = \ell_{C'}(\rho)$ for all hyperbolic structures $\rho$ on $S$, we have that $C' = \tau(C)$. 
\label{punctured torus}
\end{lemma}

\begin{proof} Since $C$ has self-intersection number one, a representative loop $\gamma$ in $C$ is the concatentation $\gamma = \gamma_0\cdot \gamma_1$ of two homotopically non-trivial and homotopically distinct (oriented) simple loops $\gamma_0$ and $\gamma_1$ on $S$.    Working in a hyperbolic structure $\rho$ on $S$ of infinite volume, so that its end is a funnel, let $C_0$ be the curve determined by $\gamma_0$ and $C_1$ the curve determined by $\gamma_1$.  As before, let $C^*_k$ be the closed geodesic representative of $C_k$ and hence the simple closed geodesic freely homotopic to $\gamma_k$.  

Note that since $\gamma_0$ and $\gamma_1$ are freely homotopic to disjoint loops on $S$, which is easily seen by performing the required cut and paste operation in a small neighborhood of their point of intersection, we have that their corresponding curves $C_0$ and $C_1$ are either disjoint or equal, as are the simple closed geodesics $C^*_0$ and $C^*_1$.  Up to relabelling, there are three possibilities, which we take in turn.

All three possibilities use the same basic argument we've used before, with minor variations.   Given the curve $C$, we find a curve $E$ so that either $E$ is disjoint from $C$ and intersects $C'$, which allows us to construct a contradiction to the assumption that the lengths of $C$ and $C'$ are equal over all hyperbolic structures on $S$ by letting the length of $E$ go to $0$, or both $C$ and $C'$ are disjoint from $E$, in which case cutting along $E$ reduces us to the case of $S$ being a three-holed sphere and we can then use the argument from Lemma \ref{three holed sphere} and the paragraphs following its proof.

The first possibility is that  $C^*_0$ and $C^*_1$ are both separating simple closed geodesics.  On $\Sigma_{1,1}$, there cannot exist two disjoint separating curves, by a straightforward Euler characteristic argument, and so we must have that $C^*_0 = C^*_1$ and  $C_0 = C_1$.   Moreover, we have that $\tau(C) = C$ and that $\tau$ preserves the orientation of (the geodesic representative of) $C$.

As $C_0$ and $C_1$ are both separating, and hence peripheral, there exists a non-separating simple curve $E$ which is disjoint from both $\gamma_0$ and $\gamma_1$, and hence disjoint from $C$.   If $E$ intersects $C'$, then we construct a sequence of hyperbolic structures on $S$ in which the length of $E$ goes to $0$ while the length of $C$ remains bounded.   However,  the length of $C'$ goes to $\infty$ as the length of the intersecting curve $E$ goes to $0$, which yields a contradiction. 

Hence in the case that both $C_0$ and $C_1$ are separating, we see that every non-separating simple curve $E$ disjoint from $C$ is also disjoint from $C'$.  Cutting $S$ along such an $E$ yields a three-holed sphere $Z$ for which the two peripheral curves that arise from $E^*$ have equal positive length.   On $Z$, we have that $C$ is a figure-eight curve, as it has a single point of self-intersection, and for all hyperbolic structures on $Z$ for which the two peripheral curves that arise from $E^*$ have equal positive length, we have that $C$ and $C'$ have equal length.  

We now follow the argument from the proof of Lemma \ref{three holed sphere}, together with the observation immediately following that proof, to conclude that  $C'$ also has self-intersection number one.  Moreover, by varying the two available lengths, we have that both $C$ and $C'$ go around the end of $Z$ determined by $C_0$ (or equivalently $C_1$) and one of the other two ends.  Since the other two ends of $Z$ have the same length over all of the relevant hyperbolic structures, we have that $\ell_C (\rho) = \ell_{C'}(\rho)$ for all hyperbolic structures $\rho$ on $S$.  

It remains only to show that $C' = \tau(C) =C$.  This however follows from the observation that on $Z$, we have that $C'$ and $C$ are both curves on $Z$ with self-intersection number one that pass around the same two ends of $Z$, and hence are the same figure-eight curve.

The second possibility is that both of the geodesics $C^*_0$ and $C^*_1$ are non-separating simple closed geodesics.  On $\Sigma_{1,1}$, there cannot exist two disjoint non-separating curves, since the complement in $S$ of one non-separating simple curve is a three-holed sphere which itself contains no non-separating simple closed curves, and so we must have that $C_0 = C_1$.   Moreover, we have that $\tau(C) = C$ and that $\tau$ preserves the orientation of (the geodesic representative of) $C$.

We note that by construction $C$ and $C_1$ are disjoint.   Moreover, given any hyperbolic structure on $S$, cutting $S$ along the simple closed geodesic representative $C_1^\ast$ of $C_1$ yields a hyperbolic structure on the three-holed sphere $Z$ for which two of the peripheral curves have equal lengths.  The hyperbolic structures on $S$ are in one-to-one correspondence with the hyperbolic structures on the three-holed sphere $Z$ for which the two  peripheral curves that arise from $C_1^\ast$ have equal positive length.  Moreover, $C$ is a figure-eight curve on $Z$.

We now note that $C'$ must also be disjoint from $C_1$, as otherwise, we can find (as we have done several times up to this point in the argument) a sequence of hyperbolic structures on $S$ in which the length of $C_1$ goes to zero, thereby forcing the length of $C'$ to go to infinity, while the length of $C$ remains bounded, which is a contradiction to our assumption that $\ell_C (\rho) = \ell_{C'} (\rho)$ for all hyperbolic structures $\rho$ on $S$.  In particular, $C'$ is a curve on $Z$.  We now follow the same argument as given in the previous case.

The third and final possibility is that $C^*_0$ is separating and $C^*_1$ is non-separating.   By Theorem \ref{haas susskind} (after doubling $\Sigma_{1,1}$ across its boundary and extending the hyperelliptic involution on $\Sigma_{1,1}$ to the hyperbolic involution on $\Sigma_2$), we have that $\tau(C_k) = C_k$ for $k=0$ and $k=1$.  However, we also have that $\tau$ preserves the orientation on $\gamma_0$ and that $\tau$ reverses the orientation on $\gamma_1$.  In particular, we must have that $\tau(C) \ne C$, because we cannot have that $\tau$ preserves the orientation on one part of $C$ and reverses the orientation on the other part of $C$.  

Since $C_0$ is peripheral, we have that $C_0$ and $C_1$ are disjoint, and hence that $C$ and $C_1$ are disjoint.   Moreover, given any hyperbolic structure on $S$, cutting $S$ along the simple closed geodesic representative $C_1^\ast$ of $C_1$ yields a hyperbolic structure on the three-holed sphere $Z$ for which two of the peripheral curves have equal lengths.  The hyperbolic structures on $S$ are in one-to-one correspondence with the hyperbolic structures on the three-holed sphere $Z$ for which the two  peripheral curves that arise from $C_1^\ast$ have equal positive length.  Moreover, $C$ is a figure-eight curve on $Z$.

We now note that $C'$ must also be disjoint from $C_1$, as otherwise, we can find (as we have done several times up to this point in the argument) a sequence of hyperbolic structures on $S$ in which the length of $C_1$ goes to zero, thereby forcing the length of $C'$ to go to infinity, while the length of $C$ remains bounded, which is a contradiction to our assumption that $\ell_C (\rho) = \ell_{C'} (\rho)$ for all hyperbolic structures $\rho$ on $S$.  In particular, $C'$ is a curve on $Z$. 

We now follow the argument from the proof of Lemma \ref{three holed sphere}, together with the observation immediately following that proof, to conclude that  $C'$ also has self-intersection number one.  Moreover, by varying the two available lengths, we have that both $C$ and $C'$ go around the end of $Z$ determined by $C_0$ and one of the other two ends.  Since the other two ends of $Z$ have the same length over all of the relevant hyperbolic structures, we have that $\ell_C (\rho) = \ell_{C'}(\rho)$ for all hyperbolic structures $\rho$ on $S$.  

It remains only to show that $C' = \tau(C)$.  This however follows from the observation that on $Z$, we have that $C'$ and $\tau(C)$ are both curves on $Z$ with self-intersection number one that pass around the same two ends of $Z$, and hence are the same figure-eight curve.
\end{proof}

Lemma \ref{punctured torus} has the following consequence.  Suppose now that $\Sigma$ is any compact, orientable surface of negative Euler characteristic which is not itself a torus with one hole, and let $T\subset \Sigma$ be an embedded torus with one hole in $\Sigma$.  By this, we mean that there is a separating curve $C$ on $\Sigma$ so that one of the components of the complement of (a simple representative from) $C$ in $\Sigma$ is a torus with one hole.   

In particular, as no boundary components of $\Sigma$ can lie in $T$, we can make the same observation as made above for hyperbolic structures on pairs of pants.  That is, every hyperbolic structure on the interior $S$ of $\Sigma$ restricts to a hyperbolic structure on $T$ for which the one end is a funnel, and conversely, every hyperbolic structure on $T$ for which the one end is a funnel extends (in a highly non-unique way) to a hyperbolic structure on $S$.  

Hence, for every hyperbolic structure $\rho$ on $S ={\rm int}(\Sigma)$, every embedded torus with one hole  $T\subset \Sigma$, and every curve $C$ in $T$ with self-intersection number one, we are in the situation that if we let $\tau_\rho$ be the hyperelliptic involution for the restriction of $\rho$ to the interior of $T$, then the curve $C' =\tau_\rho(C)$ will necessarily satisfy $\ell_C (\rho) = \ell_{C'} (\rho)$.  

Lemma \ref{genus two} and Lemma \ref{punctured torus} (and these remarks following the proof of Lemma \ref{punctured torus}) together highlight several situations in which the length function of a curve with self-intersection number one on an orientable, hyperbolic surface does not completely characterise the curve.   We group these exceptional situations together as the {\em hyperelliptic exception}, as all of them arise in one way or another from hyperellipticity.  The main result of this Section is that these are the only exceptional situations. 

{\bf Theorem \ref{one-self-intersection}.} {\em Let $S$ be the interior of a compact, orientable surface $\Sigma$ of negative Euler characteristic.  Let $C$ be a curve on $S$ with self-intersection number one.  If $C'$ is a curve on $S$ that satisfies $\ell_{C'} (\rho) = \ell_C (\rho)$ for all $\rho\in {\cal D}(\pi_1(S))$, then either $C' = C$ or we are in the case of the hyperelliptic exception, so that one of the following holds:
\begin{enumerate}
\item We have that $S = \Sigma_2$ is the closed, orientable surface of genus two with hyperelliptic involution $\tau$, the curve $C$ is contained in the interior of a pair of pants $P\subset \Sigma_2$ for which every boundary component of $P$ is a non-separating curve on $\Sigma_2$, and $C' = \tau(C)$. 
\item There exists an embedded torus with one hole $T\subset \Sigma$ containing $C$, that $\tau$ is the hyperelliptic involution restricted to $T$, and that $C' = \tau(C)$.  
\end{enumerate}}

\begin{proof} The proof of Theorem \ref{one-self-intersection} uses many of the same basic facts as does the proof of Theorem \ref{mcshane surface}, though the details are significantly different, together with the Lemmas above.  Again by using Fenchel-Nielsen coordinates and Lemma \ref{hempel} as in the proof of Theorem \ref{mcshane surface}, we see immediately that $C'$ cannot be simple.  

For the sake of concreteness, fix a hyperbolic structure $\rho_0$ on $S$.  (The proof is independent of which hyperbolic structure is chosen.)  Let $X$ be the smallest subsurface of $S_{\rho_0}$ with totally geodesic boundary containing the geodesic representative $C^\ast$ of $C$, and let $Y$ be the closure of $S_{\rho_0} \setminus X$, so that $Y$ is a (possibly disconnected) subsurface of $S_{\rho_0}$ with totally geodesic boundary.  By considering for instance a stable neighborhood of $C^\ast$, we see that there are two possibilities for $X$, namely either $X$ is a pair of pants or $X$ is a torus with one hole.   

In the former case, in which  $X$ is a pair of pants, we have already seen that $X$ has three boundary geodesics.   The curve $C'$ either intersects one of the boundary geodesics of $X$, is contained in $X$, or is contained in (a component of) $Y$.  

If $C'$ intersects a boundary geodesic of $X$, then as we have done several times, we construct a sequence of hyperbolic structures on $S$ for which the length of this boundary geodesic goes to zero, which forces the length of $C'$ to go to infinity, while leaving the length of $C$ bounded.   This contradicts our basic assumption that the lengths of $C$ and $C'$ are equal for all hyperbolic structures on $S$.

If $C'$ is contained in $X$, we have that $C' = C$ by Lemma \ref{three holed sphere}.   

If $C'$ is contained in (a component of) $Y$, then either $Y$ is itself a pair of pants, in which case $S$ is a closed surface of genus two and by Lemma \ref{genus two} we are in exceptional case (1) in the statement of the Theorem, or $Y$ contains a non-peripheral simple curve $\eta$ that intersects $C'$.  In this latter case, we  construct a sequence of hyperbolic structures on $S$, all of which are extensions of a fixed hyperbolic structure on $X$, for which the length of $\eta$ goes to zero, so that in turn the length of $C'$ goes to infinity, while the length of $C$ remains constant.  Again, this yields a contradiction to our basic assumption that the lengths of $C$ and $C'$ are equal for all hyperbolic structures on $S$. 

In the latter case, in which $X$ is a torus with a hole, we have already seen that $X$ has a single boundary geodesic, which is necessarily a separating curve on $S$, and so in this case the complement $Y$ of $X$ is connected and has positive genus.  The curve $C'$ either intersects this boundary geodesic, is contained in $X$, or is contained in $Y$.

If $C'$ intersects the boundary geodesic of $X$, then as we have done several times, we construct a sequence of hyperbolic structures on $S$ for which the length of this boundary geodesic goes to zero, which in turn forces the length of $C'$ to go to infinity, while leaving the length of $C$ bounded.   This contradicts our basic assumption that the lengths of $C$ and $C'$ are equal for all hyperbolic structures on $S$.

If $C'$ is contained in $X$, then by Lemma \ref{punctured torus}, we are in exceptional case (2) of the Theorem. 

If $C'$ is contained in $Y$, there exists a non-peripheral simple curve $\eta$ in $Y$ that intersects $C'$.  As we have done before, we  construct a sequence of hyperbolic structures on $S$, all of which are extensions of a fixed hyperbolic structure on $X$, for which the length of $\eta$ goes to zero, so that in turn the length of $C'$ goes to infinity, while the length of $C$ remains constant.  Again, this yields a contradiction to our basic assumption that the lengths of $C$ and $C'$ are equal for all hyperbolic structures on $S$. 
\end{proof}

It is not possible to generalise Theorem \ref{one-self-intersection} further to higher self-intersection number, as is demonstrated by the following example.  Let $S$ be the torus with a single hole with fundamental group $\pi_1(S) = \langle a, b\rangle$.  Horowitz \cite{horowitz} noted that the two elements $w = aba^2 b^{-1}$ and $w' = a^2 bab^{-1}$ generate non-conjugate maximal infinite cyclic subgroups and hence represent non-equal curves, have equal characters $\chi[w] = \chi[w']$, and both have two self-intersection points.

Beyond self-intersection numbers, there are remarkably few topological characterisations of pairs (or $n$-tuples) of curves which have the same character; that is, there are remarkably few conditions $P$ for which the following statement holds true: 

{\em Let $S$ be the interior of a compact, orientable surface of negative Euler characteristic.  Let $C$ be a curve on $S$ satisfying the condition $P$.  If $C'$ is a curve on $S$ that satisfies $\ell_{C'} (x) = \ell_C (x)$ for every hyperbolic structure $x$ on $S$, then either $C' = C$ or we are in one of a small number of explicitly listed, geometrically natural cases.}

An early conjectural such characterisation was that two curves have the same character if and only if they have the same intersection number with every simple curve on the surface.  The necessity of this characterisation was intuitively straightforward and proved by Leininger \cite{leininger}, who also gave examples to show of curves whose length functions are not equal but which have the same intersection number with every simple curve on the surface.  

\section{$3$-dimensional preliminaries}
\label{preliminaries}

The purpose of this Section is to present the background material on $3$-manifolds and Kleinian groups that we will need in future Sections.  Standard references for this material are Hempel \cite{hempel book} for $3$-manifold topology in general, and Maskit \cite{maskit book}, Kapovich \cite{kapovich}, and Matsuzaki and Taniguchi \cite{mt} for Kleinian groups and hyperbolic $3$-manifolds.   Given these references, we do not always provide references to the original sources.

\subsection{$3$-manifold topology}
\label{topology basics}

A compact, orientable $3$-manifold $M$ is {\em irreducible} if every embedded 2-sphere in $M$ bounds a ball in $M$.   We note that if $M$ is irreducible and has non-empty boundary, then every boundary component has positive genus. 

Let $M$ be a compact, orientable, irreducible $3$-manifold.  An orientable, embedded surface $S\subset M$ is {\em properly embedded} if $S\cap \partial M = \partial S$.  A properly embedded surface $(S, \partial S)\subset (M, \partial M)$ is {\em incompressible} if $\pi_1(S)$ is infinite and the inclusion $S\hookrightarrow M$ induces an injective map on fundamental groups.  A properly embedded surface $S\subset M$ is {\em essential} if $S$ is incompressible and not homotopic into $\partial M$.   

Similarly, a component $S$ of $\partial M$ is {\em incompressible} if the inclusion $S\hookrightarrow M$ induces an injective map on fundamental groups.   A union $S = S_1\cup \cdots \cup S_n$ of incompressible components of $\partial M$ is {\em an-annular} if there does not exist an essential annulus $A$ in $M$ with both components of $\partial A$ contained in $S$, and $M$ is {\em acylindrical} if the whole of $\partial M$ is an-annular.  

A compact, orientable, irreducible $3$-manifold $M$ is {\em atoroidal} if every incompressible torus in $M$ is homotopic into $\partial M$.  A compact, orientable, irreducible $3$-manifold $M$ has {\em atoroidal boundary} if every component of $\partial M$ has genus at least two.

A $3$-submanifold $M$ of an irreducible $3$-manifold $N$ is {\em incompressible} if $M$ is irreducible and the inclusion $M \hookrightarrow N$ induces an injective map on fundamental groups.

\subsection{Kleinian groups}
\label{kleinian basics}

A {\em Kleinian group} $\Gamma$  is a discrete subgroup of the group ${\rm Isom}^+({\mathbb H}^3) \cong {\rm PSL}_2({\mathbb C})$ of orientation-preserving isometries of the real hyperbolic $3$-space ${\mathbb H}^3$.  The action of $\Gamma$ on ${\mathbb H}^3$ extends to an action by conformal homeomorphisms (M\"obius transformations) on the Riemann sphere $\overline{\mathbb C}$, which is the boundary at infinity of ${\mathbb H}^3$.  The {\em domain of discontinuity} $\Omega(\Gamma)$ of $\Gamma$ is the largest open subset of $\overline{\mathbb C}$ on which $\Gamma$ acts properly discontinuously.  The {\em limit set} $\Lambda(\Gamma)$ is the complement of $\Omega(\Gamma)$ in $\overline{\mathbb C}$, or equivalently, the closure of the set of fixed points of infinite order elements of $\Gamma$.   We assume that all Kleinian groups in this paper are torsion-free.  

The {\em convex hull} ${\rm hull}(\Gamma)$ of $\Gamma$ is the smallest non-empty convex subset of ${\mathbb H}^3$ which is invariant under $\Gamma$.  Equivalently, the convex hull is the smallest convex subset of ${\mathbb H}^3$ containing all of the hyperbolic lines in ${\mathbb H}^3$ both of whose endpoints at infinity lie in $\Lambda(\Gamma)$.    The quotient of the convex hull is the {\em convex core} ${\rm core}(\Gamma) = {\rm hull}(\Gamma)/\Gamma$ of the hyperbolic $3$-manifold ${\mathbb H}^3/\Gamma$, which is the smallest convex submanifold of ${\mathbb H}^3/\Gamma$ whose inclusion induces a homotopy equivalence.  

A Kleinian group $\Gamma$ is {\em geometrically finite} if some, and hence every, $\varepsilon$-neighbhorhood of ${\rm core}(\Gamma)$ has finite volume, and is {\em convex co-compact} if its convex core is compact.   Equivalently, $\Gamma$ is convex co-compact if and only if either, and hence both, of its associated  $3$-manifolds $({\mathbb H}^3\cup\Omega(\Gamma))/\Gamma$ and ${\rm core}(\Gamma)$ are compact.  In this case, we can see that ${\rm core}(\Gamma)$ is naturally homeomorphic to $({\mathbb H}^3\cup\Omega(\Gamma))/\Gamma$.

\subsection{Hyperbolic structures on $3$-manifolds and deformation theory of Kleinian groups}
\label{deformation basics}

A compact, orientable $3$-manifold $M$ is {\em hyperbolizable} if there exists a (necessarily finitely generated) Kleinian group $\Gamma$ so that ${\rm int}(M)\cong {\mathbb H}^3/\Gamma$.  We refer to $\Gamma$ as a Kleinian group {\em uniformizing} $M$.   In general, if $\Gamma$ is a geometrically finite Kleinian group uniformizing $M$, then $\Omega(\Gamma)/\Gamma$ is naturally identified with a subset of $\partial M$ which is the complement of a finite collection of annuli in $\partial M$ together with all the torus components of $\partial M$.

A hyperbolizable $3$-manifold is necessarily orientable, irreducible, and atoroidal.    We can say slightly more.  For a hyperbolizable $3$-manifold $M$, every maximal ${\mathbb Z}\oplus {\mathbb Z}$ subgroup of $\pi_1(M)$ corresponds to a torus component of $\partial M$.  By this we mean that if $\Gamma$ is a Kleinian group uniformizing $M$ and if $\Theta$ is a maximal ${\mathbb Z}\oplus {\mathbb Z}$ subgroup of $\Gamma$, then all non-trivial elements of $\Theta$ are parabolic and there exists a (necessarily incompressible) torus component $T$ of $\partial M$ so that $\Theta = \pi_1(T)$ (up to conjugacy).  

A finitely generated (torsion-free) group $G$ has naturally associated to it the (possibly empty) space ${\cal D}(G)$ of all realisations of $G$ as a Kleinian group,  that is
\[ {\cal D}(G) =  \{ \rho: G\rightarrow {\rm PSL}_2({\mathbb C})\: |\: \rho\mbox{ is faithful and }\rho(G)\mbox{ is a Kleinian group} \}. \]
As before, the natural topology on ${\cal D}(G)$ comes from choosing a collection $\{ g_1,\ldots, g_p\}$ of elements of $G$ that generates and realizing ${\cal D}(G)$ as a subset of $({\rm PSL}_2({\mathbb C}))^p$ via the map $\rho\mapsto (\rho(g_1),\ldots, \rho(g_p))$.   Let 
\[ {\cal CC}(G) = \{ \rho: G\rightarrow {\rm PSL}_2({\mathbb C})\: |\: \rho(G) \mbox{ is convex co-compact}. \} \]
It is a fundamental result of J\o rgensen \cite{jorgensen} that when non-empty, ${\cal D}(G)$ is closed in the case that $G$ is the fundamental group of a compact, hyperbolizable $3$-manifold with non-empty, incompressible boundary.   

Note that ${\rm PSL}_2({\mathbb C})$ acts naturally on ${\cal D}(G)$ by conjugation, yielding the quotient ${\rm AH}(G) = {\cal D}(G)/{\rm PSL}_2({\mathbb C})$.  We will abuse notation and, where it is clear in context, blur the distinction between convergence of representations in ${\cal D}(G)$ and classes of representations in ${\rm AH}(G)$.  The connection between convergence of sequences in these two spaces is that a sequence $\{ [\rho_n ]\} \subset {\rm AH}(G)$ converges to $[\rho ] \in {\rm AH}(G)$ if and only if there exists a  sequence $\{ h_n\} \subset {\rm PSL}_2({\mathbb C})$ converging to the identity so that $\{ h_n \rho_n h_n^{-1} \}$ converges to $\rho$ in ${\cal D}(G)$.  

We have the following consequence of Mostow-Prasad rigidity.  Let $M$ be a closed, hyperbolizable $3$-manifold.  The hyperbolic structure on $M$ is unique, so that ${\rm AH}(\pi_1(M))$ consists of a single point.   In the case that $\partial M$ is non-empty, we consider hyperbolic structures on the interior ${\rm int}(M)$ of $M$.  Similar to the case of closed $3$-manifolds, suppose that $M$ is a compact, hyperbolizable $3$-manifold with $\partial M$  the union of tori.  Then we have again that ${\rm AH}(\pi_1(M))$ consists of a single point, so that again the hyperbolic structure on ${\rm int}(M)$ is unique.  

For the remainder of this note, we make the standing assumption that $M$ is a compact, hyperbolizable $3$-manifold with non-empty, incompressible, atoroidal boundary, so that $\partial M$ is non-empty, every component of $\partial M$ has genus at least two, or equivalently, so that $\pi_1(M)$ contains no ${\mathbb Z}\oplus {\mathbb Z}$ subgroup.  It is possible to extend the results of this note to compact, hyperbolizable $3$-manifolds with incompressible boundary, whose boundaries contain tori; however, considering such manifolds introduce resolvable but unpleasant complications, some similar in nature to the complications discussed in the remark following the proof of Theorem \ref{mcshane surface} for surfaces.

For such $M$, we have the following description of the structure of ${\cal CC}(\pi_1(M))$ and ${\cal D}(\pi_1(M))$; see the Introduction of Canary and McCullough \cite{canary mccullough} for a more detailed discussion.  The space ${\cal CC}(\pi_1(M))$ consists of a finite collection of disjoint open subsets of ${\cal D}(\pi_1 (M))$ parametrised by equivalence classes of pairs $(M_0, f_0)$, where $M_0$ is a compact, hyperbolizable $3$-manifold and $f_0 : M\rightarrow M_0$ is a homotopy equivalence, with the relation that $(M_0, f_0) \sim (M_1, f_1)$ if there exists a homeomorphism $g: M_0\rightarrow M_1$ with $f_1\sim g\circ f_0$.   It follows from work of Ahlfors, Bers, Kra, Maskit, Sullivan and Thurston that ${\cal CC}(\pi_1(M))$ is the interior of ${\cal D}(\pi_1(M))$.  It follows from the resolution of the Bers--Thurston Density Conjecture by Brock, Canary and Minsky \cite{brock canary minsky}, which in turn follows from their resolution of Thurston's Ending Lamination Conjecture, that $\overline{{\cal CC}(\pi_1(M))} = {\cal D}(\pi_1(M))$.

In particular, given such an $M$, there is a unique distinguished component ${\cal CC}_0 (\pi_1(M))$ of ${\cal CC}(\pi_1(M))$ so that for each $\rho\in {\cal CC}_0 (\pi_1(M))$, there is a homeomorphism $f: M\rightarrow {\rm core} (\rho(\pi_1(M)))$ satisfying $\rho = f_\ast$.  That is, the representations in ${\cal CC}_0(\pi_1(M))$ are exactly those that give rise to quotient hyperbolic $3$-manifolds naturally homeomorphic to ${\rm int}(M)$.  

In the case that $M$ is acylindrical, we have much more.  First, we have that the space of convex co-compact representations is connected, so that ${\cal CC}(\pi_1(M)) = {\cal CC}_0 (\pi_1(M))$; this follows immediately from the result of Johannson \cite{johannson} that a homotopy equivalence between acylindrical $3$-manifolds is homotopic to a homeomorphism.   Second, we have the following case of a fundamental theorem of Thurston.

\begin{theorem} [Thurston \cite{thurston acylindrical}] Let $M$ be a compact, hyperbolizable, acylindrical $3$-manifold with non-empty, incompressible, atoroidal boundary.  Then ${\rm AH}(\pi_1(M))$ is compact.
\label{thurston compactness}
\end{theorem}

We have an alternate description of the representations in the distinguished component ${\cal CC}_0 (\pi_1(M))$ of ${\cal CC}(\pi_1(M))$.  Given $M$, let $\Gamma$ be a convex co-compact Kleinian group so that ${\mathbb H}^3/\Gamma$ is homeomorphic to ${\rm int}(M)$.   Without loss of generality, assume that $\Gamma = \rho_0 (\pi_1(M))$ for some $\rho_0 \in {\cal CC}_0(\pi_1(M))$.  We can find all other representations in ${\cal CC}_0(\pi_1(M))$ by conjugating $\Gamma$ by quasiconformal homeomorphisms of $\overline{\mathbb C}$ equivariant with respect to the action of $\Gamma$; see for instance Section 3.3 of Matsuzaki and Taniguchi \cite{mt}.  Rephrased, for any $\rho\in {\cal CC}_0(\pi_1(M))$, there exists a quasiconformal homeomorphism $\omega: \overline{\mathbb C}\rightarrow \overline{\mathbb C}$ so that $\rho(\gamma) = \omega \rho_0(\gamma) \omega^{-1}$ for all $\gamma\in \Gamma$.

We will have occasion to make use of the restriction of the realisation of a finitely generated group $G$ as a Kleinian group to a finitely generated subgroup $H$ of $G$.  Specifically, we need the following Lemma, which follows immediately from the description of the representations in ${\cal CC}_0 (G)$ given above in terms of quasiconformal deformations, together with Thurston's theorem (see for instance Morgan \cite{morgan}, Proposition 7.1) that finitely generated subgroups of geometrically finite Kleinian groups with non-empty domain of discontinuity are themselves geometrically finite.

\begin{lemma} Let $N$ be a compact, hyperbolizable $3$-manifold with non-empty, incompressible, atoroidal boundary and let $M\subset N$ be an incompressible $3$-submanifold with incompressible boundary.  If $\rho\in {\cal CC}_0 (\pi_1(N))$, then the restriction of $\rho$ to $\pi_1(M)\subset \pi_1(N)$ yields an element $\rho\in {\cal CC}_0 (\pi_1(M))$.
\label{restricting submanifold}
\end{lemma}

\subsection{Curves in hyperbolizable $3$-manifolds}
\label{curves basics}

As per our standing assumption, let $M$ be a compact, hyperbolizable $3$-manifold with non-empty, incompressible, atoroidal boundary.  We mimic the definitions relating to and the basic properties of curves as given for surfaces in Section \ref{curves on surfaces}. 

The non-trivial element $c\in \pi_1(M)$ is {\em maximal} if it is not the proper power of another element of $\pi_1(M)$, or equivalently if $\langle c\rangle$ is a maximal cyclic subgroup of $\pi_1(M)$.   As with surfaces, every non-trivial element of $\pi_1(M)$ is either maximal or a proper power of a maximal element, as elements of the fundamental group of a hyperbolizable $3$-manifold are not infinitely divisible. 

A {\em curve} $C$ in $M$ is the free homotopy class corresponding to a maximal element $c\in \pi_1 (M) $.    We work throughout with {\em unoriented curves}, so that for an element $c\in \pi_1(M)$ with its corresponding free homotopy class $C$, we have that $C$ is also the curve corresponding to $c^{-1}$.  Since curves are free homotopy classes of elements of $\pi_1(M)$, each curve naturally corresponds to a conjugacy class of maximal cyclic subgroups of $\pi_1(M)$. 

The relationship of curves in a compact, hyperbolizable $3$-manifold $M$ to the boundary $\partial M$ of $M$ is more complicated than the corresponding relationship for surfaces.  Let $C$ be a curve in $M$.  The basic distinction is whether a curve $C$ is or is not freely homotopic to a curve on $\partial M$; even here, complications arise, because even in the case that $C$ is freely homotopic to a curve on $\partial M$, it may be that $C$ is freely homotopic into more than one component of $\partial M$ or that $C$ is freely homotopic to distinct curves in the same component of $\partial M$.   Moreover, we have the distinction of whether $C$ is freely homotopic to a simple or a non-simple curve on $\partial M$. 

Note that for a curve $C$ in $M$, if $C$ is homotopic to distinct simple curves in the same component of $\partial M$, the curves in $\partial M$ must be disjoint.  Moreover, if $C$ is homotopic to curves in distinct components of $\partial M$, then $C$ is homotopic to a simple curve on one boundary component if and only if $C$ is homotopic to a simple curve in every  component of $\partial M$ into which it is homotopic.  

One additional subtlety is that there exists a compact, hyperbolizable $3$-manifold $M$ and a curve $C$ in $M$ so that $C$ is not itself freely homotopic onto $\partial M$, but some proper power of $C$ is freely homotopic to a simple curve on $\partial M$.  However, we will not consider such curves in this note.

Let $\rho\in {\cal CC}_0 (\pi_1(M))$, so that ${\rm int}(M)$ is homeomorphic to ${\mathbb H}^3/\rho(\pi_1(M))$.  Because we have assumed that $M$ has atoroidal boundary,  $\pi_1(M)$ has no ${\mathbb Z}\oplus {\mathbb Z}$ subgroups, and so $\rho(\pi_1(M))$ has no parabolic elements.   In particular, each curve $C$ in $M$ is freely homotopic to a (unique) closed geodesic $C^\ast$ in ${\mathbb H}^3/\rho(\pi_1(M))$.  There is a natural correspondence between maximal cyclic subgroups of $\pi_1(M)$, or equivalently curves in $M$, on the one hand and closed geodesics in ${\mathbb H}^3/\rho(\pi_1(M))$ for some, and hence for every, hyperbolic structure $\rho\in {\cal CC}_0(\pi_1(M))$ on the other hand.   

Unlike the case of hyperbolic surfaces, there is no way to associate adjectives such as simple to a curve $C$ in $M$, as the simplicity in ${\mathbb H}^3/\rho(\pi_1(M))$ of $C^\ast$ depends sensitively on the hyperbolic structure induced by $\rho$.    In fact, by the discussion above of the structure of ${\cal CC}(\pi_1(M))$, the relationship of a curve $C$ in $M$ to $\partial M$ is also problematic, as the topological type of the hyperbolic $3$-manifold varies over the components of ${\cal CC}(\pi_1(M))$.  We resolve these issues by starting with the  $3$-manifold $M$ and restricting our attention to the (convex co-compact) hyperbolic structures on ${\rm int}(M)$, which are precisely the elements of ${\cal CC}_0 (\pi_1(M))$.  

Each curve $C$ in $M$ has associated to it a map $\ell_C: {\cal CC}_0(\pi_1(M))\rightarrow {\mathbb R}$ given by setting $\ell_C (\rho)$ to be the (real) length of the associated closed geodesic $C^\ast$ in the hyperbolic $3$-manifold ${\mathbb H}^3/\rho(\pi_1(M))$.   (We note that it is possible to consider these arguments using the complex length associated to loxodromic elements and their corresponding closed geodesics, but for the purposes of this note, considering real length is sufficient.)

We make use of the following result, which is an immediate consequence of results from Sections 2 and 3 of Maskit \cite{maskit-parabolics}, expressed in the language above.

\begin{proposition} Let $M$ be a compact, hyperbolizable $3$-manifold with non-empty, incompressible, atoroidal boundary and let $C$ be a curve in $M$ freely homotopic to a simple curve on $\partial M$.  There exists a sequence $\{ \rho_n\}\subset {\cal CC}_0 (\pi_1(M))$ so that $\ell_C (\rho_n)\rightarrow 0$ as $n\rightarrow\infty$.
\label{maskit-squeezing}
\end{proposition}

Maskit's proof proceeds by constructing a sequence of quasiconformal deformations of $\partial M$ for which the length of $C$ on $\partial M$ goes to $0$.  These deformations necessarily give rise to a sequence $\{\rho_n\}\subset {\cal CC}_0(M)$ for which the length of $\rho_n(C)$ on $\partial M$ goes to $0$, which in turn forces $\ell_C(\rho_n)\rightarrow 0$.  We note here that the main focus of Maskit's work is to then establish the convergence of this sequence, which  requires additional hypotheses on $M$; however, we do not need here the convergence of the sequence, and so making use of the first part of Maskit's construction suffices to yield the desired sequence.  

In fact, Maskit's argument is but one variant of an observation that follows immediately from the identification of ${\cal CC}_0(\pi_1(M))$ with the Teichm\"uller space ${\cal T}(\partial M)$ of (marked) hyperbolic structures on $\partial M$.   Given a curve $C$ in $M$ homotopic to a simple curve (again called $C$) on $\partial M$, choose a sequence $\{ \rho_n\}\subset {\cal T}(\partial M)$ so that the length of $\rho_n (C)\rightarrow 0$, where here we are using the natural hyperbolic length on $\partial M = \Omega(\rho_n(\pi_1(M)))/\rho_n(\pi_1(M))$.   The existence of such a sequence follows immediately from the simplicity of $C$ on $\partial M$ and arguments similar to those given in the previous Section.  

We then use McMullen's formulation (see Proposition 6.4 and Corollary 6.5 of McMullen \cite{mcmullen-iteration}) of  Bers' inequality, which states that geodesics on $\Omega(\rho_n(\pi_1(M)))/\rho_n(\pi_1(M))$ which are short, are then homotopic to short geodesics in the hyperbolic $3$-manifold ${\mathbb H}^3/\rho_n(\pi_1(M))$, and so correspond to elements of $\rho_n(\pi_1(M))$ that are nearly parabolic.  In particular, we have that $\ell_C (\rho_n)\rightarrow 0$.    (Implicit in this latter argument is the use of the assumptions that $\partial M$ is incompressible and that we are working with convex co-compact representations to imply that the restrictions of the $\rho_n$ to the fundamental groups of the components of $\partial M$ yield quasifuchsian groups.)

\section{Topological joinery}
\label{joinery}

In this Section, we present the basic topological constructions that underlie the arguments we use to prove Theorem \ref{separating-by-length}.   The main result that underlies the discussion in this Section is Thurston's geometrisation theorem for Haken $3$-manifolds.

\begin{theorem} A compact, orientable, irreducible, atoroidal $3$-manifold $M$ with non-empty boundary is hyperbolizable. 
\label{thurston geometrisation}
\end{theorem}

With Theorem \ref{thurston geometrisation} in hand, we begin with the general discussion of the constructions we consider.  Let $M_1$ and $M_2$ be compact, hyperbolizable $3$-manifolds, so that in particular both $M_1$ and $M_2$ are orientable, irreducible and atoroidal.  Assume  that there exist components $S_1\subset\partial M_1$ and $S_2\subset\partial M_2$ whose genera satisfy ${\rm genus}(S_1) = {\rm genus}(S_2)\ge 2$.  Suppose that $S_1$ is incompressible in $M_1$ and that $S_2$ is incompressible and an-annular in $M_2$.  We allow the possibility that the $\partial M_k$ contain components beyond the $S_k$ for both $k=1$ and $k=2$.

Let $f: S_1\rightarrow S_2$ be any (orientation-reversing) homeomorphism.  We form a new compact, orientable $3$-manifold $N$ by gluing $M_1$ and $M_2$ along $S_1$ and $S_2$ using $f$;  that is, we take the disjoint union of $M_1$ and $M_2$ and then form $N$ by identifying $x\in S_1$ with $f(x)\in S_2$ inside this disjoint union.  Inside $N$, there is a distinguished surface, namely the image $S$ of $S_2 = f(S_1)$.  For ease of notation, we write $N = M_1\cup_f M_2$. 

Alternatively, we can consider the case where $S_1$ and $S_2$ are distinct components of the boundary of a single $3$-manifold $M$, all satisfying the same hypotheses as the $M_k$ above.  In this case, we glue $S_1$ to $S_2$ via $f$ to obtain a $3$-manifold $N = M \cup_f$.  As we will not use this case of the general construction to any significant extent, we present the proofs for the case above, noting that similar arguments apply in this case as well. 

The main question we consider in this Section is to determine conditions on $M_1$ and $M_2$, and on $S_1$ and $S_2$, under which $N$ is hyperbolizable and under which $N$ is acylindrical.  We will develop some finer constructions as well. 

We first note the straightforward fact that $N$ is irreducible, which follows directly from the incompressibility of the $S_k$ in $M_k$ by a standard innermost disc argument; we give a sketch of the proof here. Let $\Sigma$ be an embedded $2$-sphere in $N$.   Isotope $\Sigma$ so that $\Sigma\cap S$ is the finite union of disjoint, simple, closed loops.  Let $\gamma$ be an innermost one of these loops on $\Sigma$, meaning that one of the components of $\Sigma -\gamma$ contains no component of $\Sigma\cap S$.   Note that $\gamma$ bounds a closed disc $D$ in $\Sigma$, namely the closure in $\Sigma$ of the component of $\Sigma -\gamma$ which contains no component of $\Sigma\cap S$.   Since the interior of $D$ is disjoint from $S$, we have that $D$ is contained in $M_k$ for either $k=1$ or $k=2$.  Since $S_k$ is incompressible in $M_k$ by assumption, we see that $\partial D$ is a homotopically trivial loop in $M_k$.  So, we can isotope $D$ into $S_k$ and thereby get rid of $\gamma$.   Repeating this argument for each loop in $\Sigma\cap S$ in turn, working outward from innermost loops, we can isotope $\Sigma$ into either $M_1$ or $M_2$.  Since both $M_k$ are irreducible, we see that $\Sigma$ necessarily bounds a $3$-ball in $M_k$ and hence in $N$.  A similar argument shows that both of the $M_k$ are incompressible in $N$. 

We next note the equally straightforward fact that $N$ is atoroidal, which follows from the incompressibility of $S_1$ in $M_1$ and the incompressibility and an-annularity of $S_2$ in $M_2$; again, we provide a sketch of the proof.   Let $T$ be an incompressible torus in $N$, and isotope $T$ so that $S\cap T$ is the finite union of disjoint, simple, closed loops.  Again performing an innermost disc argument, we can isotope away all the loops in $S\cap T$ which bound a disc in either $S$ or $T$, and thus we can assume that all the loops in $S\cap T$ are homotopically non-trivial loops on both $S$ and $T$.  Since both $S$ and $T$ are embedded surfaces in $N$, the loops in $S\cap T$ are parallel on $T$.  Moreover, since $S$ separates $N$, there must be an even number of loops in $S\cap T$. 

If $S\cap T$ is empty, then $T$ is contained in $M_k$ for either $k =1$ or $k =2$, and thus by the atoroidality of the $M_k$, we have that $T$ is then homotopic in $M_k$ into $\partial M_k$.   Since $\partial N = (\partial M_1 \setminus \{ S_1\}) \cup (\partial M_2 \setminus \{ S_2\})$ and since neither $S_k$ is a torus, we then have that $T$ is homotopic in $N$ into $\partial N$. 

It remains only to consider the case that $S\cap T$ is non-empty.  Consider the closure $A$ of a component of $T \setminus \{ S\cap T\}$ contained in $M_2$.  Since the loops in $S\cap T$ are homotopically non-trivial on $T$, we see that $A$ is an annulus.  Since $T$ is incompressible in $N$, the boundary loops of $A$ are homotopically non-trivial loops in $\partial M_2$, and so $A$ is an incompressible annulus in $M_2$.  Since $S_2$ is an-anular in $M_2$, we can homotope $A$ into $\partial M_2$, and hence into $M_1$.   Doing this for every other component of $T \setminus \{ S\cap T\}$ starting from $A$, we can homotope all of $T$ into $M_1$.  Since $M_1$ is atoroidal, we can homotope $T$ into a toroidal component of $\partial M_1$, which is also a toroidal component of $\partial N$. 

Hence, we have shown the following.

\begin{theorem} Let $M_1$ and $M_2$ be compact, hyperbolizable $3$-manifolds with non-empty, incompressible, atoroidal boundary.  Assume that there exist components $S_k$ of $\partial M_k$ so that ${\rm genus}(S_1) = {\rm genus}(S_2)\ge 2$ and so that $S_2$ is an-annular in $M_2$.  Let $f: S_1\rightarrow S_2$ be any (orientation-reversing) homeomorphism.  The $3$-manifold $N = M_1\cup_f M_2$ is a compact, hyperbolizable $3$-manifold with incompressible, atoroidal boundary, and each $M_k$ is an incompressible $3$-submanifold of $N$. 
\label{basic construction}
\end{theorem}

We note the possibility that $\partial M_k = S_k$ for either $k=1$ or $k=2$, or both, and hence that $N$ may be a closed 3-manifold.  However, this case will not occur in the arguments we make below.  

We are now in the position of being able to develop the machinery of topological joinery we will use in the following Sections.   We begin by noting the existence of the basic pieces we will use in our constructions.

\begin{theorem} [Section 3 of Fujii \cite{fujii}] For each $\sigma\ge 2$, there exists a compact, hyperbolizable, acylindrical $3$-manifold $M_\sigma$ for which $\partial M_\sigma$ is a closed, orientable surface of genus $\sigma$. 
\label{fujii}
\end{theorem}

In fact, Fujii shows that for each $\sigma\ge 2$, there are infinitely many such $3$-manifolds.   We note that this is one particular instance of a collection of more general constructions.   Unfortunately, there is no good extant survey of all ways of constructing hyperbolic $3$-manifolds.  Such constructions go back at least to the following Theorem of Myers.

\begin{theorem} [Myers \cite{myers}] Let $M$ be a compact orientable $3$-manifold, so that no component of $\partial M$ has positive Euler characteristic.  Then, there exists a link $L\subset {\rm int}(M)$ so that $M^0 = M \setminus {\rm nbhd}(L)$ is Haken, atoroidal and acylindrical with incompressible boundary.
\label{myers}
\end{theorem}

Looking ahead to the end of the paper, Theorem \ref{myers} and Theorem \ref{comar dehn} allow for the construction of hyperbolizable, acylindrical $3$-manifolds in abundance. 

The following Lemma can be viewed as describing how to cap off components of the boundary of a compact, hyperbolizable $3$-manifold without introducing essential annuli.

\begin{lemma} Let $M$ be a compact, hyperbolizable $3$-manifold with non-empty, incompressible, atoroidal boundary, and let $S$ be a component of $\partial M$.  Let $P$ be a compact, hyperbolizable, acylindrical $3$-manifold with non-empty, incompressible, atoroidal boundary and let $B$ be a component of $\partial P$ satisfying ${\rm genus}(B) = {\rm genus}(S)$.  Let $f: S\rightarrow B$ be any orientation-reversing homeomorphism, and let $Q = M\cup_f P$.   Then $Q$ is a compact, hyperbolizable $3$-manifold with incompressible, atoroidal boundary, and there does not exist an essential annulus joining a component of $\partial M \setminus \{ S\}$ to a component of $\partial P\setminus \{ B\}$.  
\label{capping off single}
\end{lemma}

\begin{proof} The existence of $P$ follows immediately from Theorem \ref{fujii}.  By Theorem \ref{basic construction}, we see that $Q$ is hyperbolizable, and that $\partial Q = (\partial M \setminus \{ S\}) \cup (\partial P \setminus \{ B\})$ is necessarily incompressible and atoroidal.  It remains only to consider possible essential annuli in $Q$.  We use an argument similar to the argument given above showing that the $3$-manifolds produced as in Theorem \ref{basic construction} are atoroidal.  

So, suppose that there exists an essential annulus $A$ in $Q$ with $\partial A = a_0\cup a_1$, where $a_0\subset \partial M \setminus \{S\}$ and $a_1\subset\partial P \setminus \{ B\}$.  Homotop $A$ in $Q$ so that $A$ intersects the incompressible surface $B = f(S)$ in a collection of disjoint, simple, homotopically non-trivial closed loops.  Since $A$ is incompressible, these loops are non-trivial on $B$.  Taking the  loop closest to $a_1$ along $A$ then yields an essential annulus in $P$, contrary to the assumption that $P$ is acylindrical. 
\end{proof}

We note that the constructed $3$-manifold $Q$ will contain any pre-existing essential annuli joining components of $\partial M \setminus \{ S\}$ to one another, although all such annuli must be contained entirely in $M$.

We next show that there exist $3$-manifolds $M$ with a single essential annulus, where we have control over how the boundary components of this annulus intersect $\partial M$.  We do this in two steps, one for the case of a separating annulus and one for a non-separating annulus.  

Recall that up to homeomorphism, a curve $C$ on a closed, orientable surface $S$ is determined either by the property of being non-separating (as all simple, non-separating curves are equivalent up to homeomorphisms of the surface) or by the genera of the components of $S \setminus \{ C\}$ (as two simple, separating curves are equivalent up to homeomorphism of the surface if and only if the genera of their complements are equal).

\begin{lemma} Given $\sigma\ge 2$, there exists a compact, hyperbolizable $3$-manifold $M$ with non-empty, incompressible, atoroidal boundary so that $\partial M = S_1\cup S_2$ with ${\rm genus}(S_1) = \sigma$, there exists an essential annulus $A\subset M$ with one boundary component in $S_1$ and one boundary component in $S_2$, the boundary component in $S_1$ is non-separating, and every essential annulus in $M$ is homotopic to $A$. 
\label{capping non separating}
\end{lemma}

\begin{proof}  By Theorem \ref{fujii}, there exists a compact, hyperbolizable, acylindrical $3$-manifold $P$ so that $\partial P$ consists of a single incompressible surface $S$ of genus $\sigma + 1$.  Let $a$ and $b$ be disjoint, simple, closed loops on $S$ so that each of $a$ and $b$ is separating and so that $S \setminus \{ a\cup b\}$ consists of three surfaces, two of genus $1$ with a single boundary component and one of genus $\sigma-1$ with two boundary components.  Let $A$ and $B$ be embedded annular neighborhoods of $a$ and $b$, respectively, with $\partial A = \{ a_0, a_1\}$ and $\partial B = \{ b_0, b_1\}$.  Choose the labels so that $a_0$ and $b_0$ bound the surface of genus $\sigma -1$ with two boundary components.  

Let $f: A\rightarrow B$ be a (orientation-reversing) homeomorphism so that $f(a_0) = b_0$ and $f(a_1) = b_1$.  The manifold $M = P \cup_f$ that results from $P$ by gluing $A$ to $B$ via $f$ is then a compact, orientable $3$-manifold with $\partial M$ consisting of two surfaces, one of genus two and one of genus $\sigma$, and $M$ contains an essential annulus which is the image of $A = f(B)$ under the gluing.  

By construction, this essential annulus $A\subset M$ has one boundary component in each component of $\partial M$, and the boundary component of $A$ in the component of $\partial M$ having genus $\sigma$ is non-separating. 

The same argument as given above shows that $M$ is irreducible and atoroidal, hence hyperbolizable.  It remains only to show that $M$ contains only the single essential annulus $A$.  However, this follows directly from the assumption that $P$ is acylindrical, along with the same style of argument given several times already.
\end{proof}

\begin{lemma} Given $\sigma\ge 2$ and $1\le k < \sigma$, there exists a compact, hyperbolizable $3$-manifold $M$ with non-empty, incompressible, atoroidal boundary so that $\partial M = S_1\cup S_2$ with ${\rm genus}(S_1) = \sigma$,  there exists an essential annulus $A\subset M$ with one boundary component in $S_1$ and one boundary component in $S_2$, the genera of the components of $S_1 \setminus \partial A$ are $k$ and $\sigma -k$, and every essential annulus in $M$ is homotopic to $A$. 
\label{capping separating}
\end{lemma}

\begin{proof} By Theorem \ref{fujii}, there exist compact, orientable, acylindrical $3$-manifolds $M_1$ and $M_2$ so that $\partial M_k$ consists of a single incompressible surface $S_k$ of genus $\sigma$.  Let $a_k$ be a simple, closed loop on $S_k$ so that the two components of $S_k \setminus \{ a_k\}$ have genera $k$ and $\sigma -k$, and let $A_k$ be an annular neighborhood of $a_k$.  Label the boundary components of $A_k$ as $b_k$ and $c_k$, where $b_k$ bounds the component of $S_k \setminus \{ a_k\}$ of genus $k$.  

Let $f: A_1\rightarrow A_2$ be a (orientation-reversing) homeomorphism satisfying $f(b_1) = c_2$ and $f(b_2) = c_1$.  The manifold $M = M_1 \cup_f M_2$ that results from $M_1$ and $M_2$ by gluing $A_!$ to $A_2$ via $f$ is then a compact, orientable $3$-manifold with $\partial M$ consisting of two surfaces, both of genus $\sigma$, so that $M$ contains an essential annulus which is the image of $A = f(B)$ under the gluing.  

By construction, this essential annulus has one boundary component in each component of $\partial M$ and both components of $\partial A$ decompose their respective boundary components into two subsurfaces of genera $k$ and $\sigma -k$. 

The same argument as given above shows that $M$ is irreducible and atoroidal, hence hyperbolizable.  It remains only to show that $M$ contains only the single essential annulus $A$.  However, this follows directly from the assumption that $M_1$ and $M_2$ are acylindrical, along with the same style of argument given several times already.
\end{proof}

We note here that the proof of the uniqueness of the essential annulus in both Lemmas \ref{capping non separating} and \ref{capping separating} can also be shown either as an immediate consequence of the characteristic submanifold theory of Jaco--Shalen and Johannson (see Canary and McCullough \cite{canary mccullough} for a discussion of the characteristic submanifold theory as it specifically relates to hyperbolizable $3$-manifolds) or by considering the relationship between essential annuli and intersections of subgroups stabilizing components of the domain of discontinuity of a Kleinian group uniformizing $M$, using arguments of Maskit \cite{maskit intersection}.  For the characteristic submanifold theory argument, the essential point is that the base $3$-manifolds for these two constructions are assumed to be acylindrical.  Therefore, the characteristic submanifold of the glued manifold in both cases is the solid torus which results from thickening the annulus along which the gluing is done, and all essential annuli in $M$ are homotopic into this solid torus.  

We are now ready to bring these Lemmas together to show that we can isolate any curve in $M$ homotopic to a simple curve on $\partial M$ from all other curves in $M$ by embedding $M$ into a larger hyperbolizable $3$-manifold $N$.

\begin{proposition} Let $M$ be a compact, hyperbolizable $3$-manifold with non-empty, incompressible, atoroidal boundary and let $C$ be a curve in $M$ freely homotopic to a simple curve in $\partial M$.  Assume that $M$ is not an $I$-bundle over a surface.  There exists a compact, hyperbolizable, acylindrical $3$-manifold $N$ with non-empty, incompressible, atoroidal boundary so that $N$ contains $M$ as an incompressible $3$-submanifold, the curve $C$ is freely homotopic to a unique simple curve on $\partial N$, and if $C' \ne C$ is any other curve in $M$, then $C'$ is not freely homotopic into $\partial N$. 
\label{final capping off}
\end{proposition}

\begin{proof} Write $\partial M = S_1\cup \cdots \cup S_n$ and recall that by assumption, each $S_k$ satisfies ${\rm genus}(S_k)\ge 2$.   By relabelling if necessary, let $S_1$ be a component of $\partial M$ into which $C$ is freely homotopic; if $C$ is freely homotopic into more than one component of $\partial M$, then choose one to be $S_1$.

By Theorem \ref{fujii}, there exists a collection $M_2,\ldots, M_n$ of compact, hyperbolizable, acylindrical $3$-manifolds, each with connected, incompressible boundary, each of which satisfies ${\rm genus}(S_k) = {\rm genus}(\partial M_k)$ for $2\le k\le n$.  Apply Lemma \ref{capping off single} $n-1$ times to construct a compact, hyperbolizable $3$-manifold $P$ so that $\partial P = S_1$ is incompressible in $P$.   Note that $C$ remains freely homotopic into $\partial P$.

Even accepting the possibility that $C$ is freely homotopic to multiple curves on $\partial P = S_1$, we know that should this occur, these curves in $S_1$ are disjoint.  (We also know that there cannot exist two essentially different homotopies of $C$ to the same curve on $\partial M$, as this would give rise to an essential torus in $M$ that cannot exist.)  So, as above, choose one.  We now apply either Lemma \ref{capping separating} or Lemma \ref{capping non separating} as appropriate to construct a compact, hyperbolizable, acylindrical $3$-manifold $N$ so that $C$ is freely homotopic to a unique curve in $\partial N$ but no other curve in $M$ is homotopic into $\partial N$.

As in the proof of Theorem \ref{basic construction}, the incompressibility of $\partial M$ implies that $M$ is an incompressible $3$-submanifold of $N$.  To see that $N$ is acylindrical, assume otherwise.  The boundary curves of any essential annulus in $N$ are disjoint curves in $\partial N$.  However, by the uniqueness of the essential annulus in the last stage of the construction of $N$ immediately implies that no such essential annulus can exist.

It remains only to show that if $C'$ is any other curve in $M$ with $C'\ne C$, then $C'$ is not homotopic into $\partial N$.  Given how $N$ was constructed, this means we need only show that $C'$ is not homotopic into $\partial N = S_1$, because if $C'$ is not homotopic into $\partial M$, then clearly $C'$ is not homotopic into $\partial N$.  So, we can assume that $C'$ is homotopic into $S_1$.  However, by assumption, we have that $C'\ne C$.  So, the uniqueness of the essential annulus in the last stage of the construction of $N$ again immediately implies that $C'$ cannot be homotopic into $\partial N$, and we are done.
\end{proof}

\section{Proof of Theorem \ref{separating-by-length}}
\label{proof main theorem}

The purpose of this Section is to complete the proof of Theorem \ref{separating-by-length}.   The key technical Lemma in this Section is a direct consequence of Theorem \ref{thurston compactness}, which can be viewed as a partial extension of Lemma \ref{hempel} to $3$-manifolds.

\begin{lemma} Let $M$ be a compact, hyperbolizable, acylindrical $3$-manifold with non-empty, incompressible, atoroidal boundary, and let $C$ be a curve in $M$.  If $C$ is not homotopic to a simple curve in $\partial M$, there exists a constant $K = k(M,C) >0$ so that $\ell_C (\rho) \ge K$ for all $\rho\in {\cal CC}(\pi_1(M))$.
\label{acylindrical-bound-element}
\end{lemma}

\begin{proof} We prove the Lemma by contradiction.  Assume that no such constant $k(M,C)$ exists, so that there exists a sequence $\{\rho_n\}\subset {\cal CC}(\pi_1(M))$ for which $\ell_C (\rho_n) \rightarrow 0$ as $n\rightarrow\infty$.  By Theorem \ref{thurston compactness}, we see (by extracting the convergent subsequence using Theorem \ref{thurston compactness} and lifting back to ${\cal D}(\pi_1(M))$) that the sequence $\{ \rho_n\}$ has a subsequence, again denoted $\{ \rho_n\}$, so that $\{ \rho_n\}$ converges to $\rho\in {\cal D}(\pi_1(M))$.    Since $\ell_C(\rho_n) \rightarrow 0$, standard deformation theory arguments using J\o rgensen's inequality imply that $\rho(C)$ cannot be trivial and hence we see that $\rho(C)$ must be parabolic.  

We need the following definition.  Let $\Gamma$ be a torsion-free Kleinian group and let $\Phi \subset \Gamma$ be a maximal, purely parabolic subgroup.  There then exists a {\em horoball} associated to $\Phi$, which is an open Euclidean ball $H_\Phi\subset {\mathbb H}^3$ invariant under the action of $\Phi$, so that $\partial H_\Phi$ intersects the Riemann sphere $\overline{\mathbb C}$ in a single point which is the common fixed point of all of the non-trivial elements of $\Phi$.    As a standard consequence of the Margulis Lemma, there exists a collection ${\cal H}$ of disjoint horoballs invariant under the action of $\Gamma$ so that there exists a horoball associated to each maximal, purely parabolic subgroup of $\Gamma$ and each horoball is associated to such a subgroup.    In general, the quotient $N_\Gamma = ({\mathbb H}^3 \setminus \cup_{H\in {\cal H}} H)/\Gamma$ is then a $3$-manifold with boundary, where the boundary is a collection of open annuli and/or tori.  

For the group $\rho(\pi_1(M))$, the assumption that $M$ has atoroidal boundary implies that $\pi_1(M)$ contains no ${\mathbb Z}\oplus {\mathbb Z}$ subgroups, and so the boundary components of $N_{\rho(\pi_1(M))}$ are open annuli corresponding to the conjugacy classes of maximal parabolic subgroups of $\rho(\pi_1(M))$.  By McCullough \cite{mccullough}, there exists a compact, hyperbolizable $3$-submanifold $P\subset N_{\rho(\pi_1(M))}$ so that the inclusion of $P$ into $N_{\rho(\pi_1(M))}$ is a homotopy equivalence and each component of $\partial N_{\rho(\pi_1(M))}$ intersects $\partial P$ in a single incompressible annulus.   In particular, we see immediately that $P$ must be acylindrical and have non-empty, incompressible, atoroidal boundary.

By construction, $\rho(C)$ is peripheral in $N_{\rho(\pi_1(M))}$ and hence in $P$, and in particularly $C$ is freely homotopic to a simple curve on $\partial P$; the simplicity follows immediately from the assumption that $C$ is maximal.  

Since $M$ and $P$ are compact, irreducible $3$-manifolds with isomorphic fundamental groups, they are homotopy equivalent.  Since $M$ is acylindrical, we can apply the result of Johannson \cite{johannson} to see that $M$ and $P$ are in fact homeomorphic, and so $C$ is freely homotopic to a simple curve in $\partial M$.  This contradiction completes the proof of the Lemma.
\end{proof}

We are now ready to prove Theorem \ref{separating-by-length}.

\begin{proof} [Proof of Theorem \ref{separating-by-length}] We are given a compact, hyperbolizable $3$-manifold $M$ with non-empty, incompressible, atoroidal boundary and a curve $C$ in $M$ freely homotopic to a simple curve on $\partial M$.  Let $C'$ be a curve in $M$ satisfying $C'\ne C$.  

By Lemma \ref{final capping off}, there exists a compact, hyperbolizable, acylindrical $3$-manifold $N$ with non-empty, incompressible, atoroidal boundary so that $N$ contains $M$ as an incompressible $3$-submanifold, so that $C$ is freely homotopic to a  simple curve on $\partial N$, and so that $C'$ is not freely homotopic into $\partial N$. 

By Proposition \ref{maskit-squeezing}, there exists a sequence $\{ \rho_n\}\subset {\cal CC}(\pi_1(N))$ so that $\ell_C (\rho_n)\rightarrow 0$ as $n\rightarrow \infty$.  However, by Proposition \ref{acylindrical-bound-element}, we have that there exists a constant $K = k(N,C) >0$ so that $\ell_{C'} (\rho_n)\ge K$ for all $n$.  Restricting to $\pi_1(M)$, as in Lemma \ref{restricting submanifold}, we then have a sequence $\{ \rho_n\}\subset {\cal CC}_0 (\pi_1(M))$ so that $\ell_C (\rho_n)\rightarrow 0$ as $n\rightarrow \infty$ and $\ell_{C'} (\rho_n)\ge K$ for all $n$.  This contradicts the assumption that $\ell_C (\rho) = \ell_{C'} (\rho)$ for all $\rho\in {\cal CC}_0(\pi_1(M))$.

The case that $M$ is an $I$-bundle over a surface follows directly from the discussion in Section \ref{curves on surfaces}.  
\end{proof}

Also, it is as yet unknown whether there exists a direct proof of Lemma \ref{acylindrical-bound-element} for $3$-manifolds containing essential annuli, which would not then require the topological joinery discussed in Section \ref{joinery} that is used to reduce the general case to the acylindrical case. 

We spend the remainder of this Section engaging in some speculation.   We start by noting that it is not possible to make the constant $k(M,C)$ appearing in Lemma \ref{acylindrical-bound-element} independent of $M$.  To show this, we make use of the following simple case of Comar's variant of the Hyperbolic Dehn Filling Theorem.

\begin{theorem} [see Comar \cite{comar}] Let $M$ be a compact, hyperbolizable $3$-manifold whose incompressible boundary $\partial M = S\cup T$ is the union of two surfaces, a torus $T $and a surface $S$ of genus at least two. Let $N= {\mathbb H}^3/\Gamma$ be a geometrically finite hyperbolic $3$-manifold and let $\psi: {\rm int}(M)\rightarrow N$ be an orientation-preserving homeomorphism.  Further assume that every parabolic element of $\Gamma$ lies (up to conjugacy) in the rank-two parabolic  subgroup corresponding to $\pi_1(T)$. Let $(m, \ell)$ be a meridian-longitude basis for $T$, and let $(p_n, q_n) \rightarrow\infty$ be a divergent sequence of  pairs of relatively prime integers.

Then, for all sufficiently large $n$, there exists a representation $\beta_n\colon\Gamma\rightarrow {\rm PSL}_2 ({\mathbb C})$ with discrete image such that
\begin{enumerate}
\item $\beta_n(\Gamma)$ is a geometrically finite Kleinian group without parabolic elements uniformizing the $3$-manifold $M (p_n,q_n)$ obtained from $M$ by performing $(p_n, q_n)$ Dehn surgery along $T$;
\item the kernel of $\beta_n\circ\psi_*$ is normally generated by $\{ m^{p_n} \ell^{q_n}\}$; and
\item $\{ \beta_n\}$ converges to the identity representation of $\Gamma$. 
\end{enumerate}
Moreover, if $i_n\colon M\rightarrow M(p_n, q_n)$ denotes the inclusion map, then for each $n$, there exists an orientation-preserving homeomorphism 
\[ \psi_n\colon {\rm int}(M (p_n, q_n)) \rightarrow {\mathbb H}^3/\beta_n(\Gamma)\] 
such that $\beta_n\circ \psi_*$ is conjugate to $(\psi_n)_*\circ (i_n)_*$.
\label{comar dehn}
\end{theorem}

As we have already seen, there are many possible hyperbolic structures on $3$-manifolds whose boundary is not the (possibly empty) union of tori, and so it is not clear {\em a priori} how to get a hyperbolic structure on the Dehn surgered manifold $M(p,q)$.  One consequence of Theorem \ref{comar dehn} is that the hyperbolic structure on the original manifold $M$ is used to then impose a hyperbolic structure on $M(p,q)$; this follows from using $\Gamma$ and its images under the $\beta_n$.   The other important consequence of Dehn surgery is that the length of the closed geodesic in the Dehn surgered manifold $M(p,q)$ homotopic to the core curve of the solid torus glued to $M$ to form $M(p,q)$ goes to $0$ as $(p,q)$ moves farther from $(0,0)$.  These are the two facts that we need to prove the following Lemma.

\begin{proposition}  For $n\ge 1$, there exist compact, hyperbolizable $3$-manifolds $M_n$, each with non-empty, incompressible, connected, atoroidal boundary; hyperbolic structures $x_n$ on ${\rm int}(M_n)$; and curves $C_n$ in $M_n$ so that $C_n$ is not homotopic into $\partial M_n$ and ${\rm \ell}_{C_n}(x_n)\rightarrow 0$. 
\label{curves getting short}
\end{proposition}

\begin{proof} Let $M$ be a compact, hyperbolizable, acylindrical $3$-manifold whose boundary is the union of two surfaces, a torus $T$ and a surface $S$ of genus at least two.  Choose a meridian-longitude basis $(m, \ell)$ for $T$, and let $(p_n, q_n) \rightarrow\infty$ be a divergent sequence of  pairs of relatively prime integers.  Let $M(p_n,q_n)$ be the result of performing $(p_n, q_n)$ Dehn surgery on $M$, as described in Theorem \ref{comar dehn}.

We have already that there are natural hyperbolic structures on the ${\rm int}(M_n)$, given a hyperbolic structure on ${\rm int}(M)$, so that the lengths of the core curves $C_n$ of the Dehn surgered manifolds are going to $0$.  The only thing remaining to show is that the $C_n$ are not homotopic to simple curves in $\partial M_n$.  In fact, we can see that $C_n$ is not homotopic to any curve in $\partial M_n$, using the same sorts of arguments as were given earlier, as such a homotopy would allow for the construction on an essential annulus in $M\subset M(p_n,q_n)$ with one boundary component of the annulus lying in $T$ and the other lying in $S$.  
\end{proof}

Lemma \ref{acylindrical-bound-element} raises the interesting question of the extent to which the bound $k(M,C)$ can be made independent of the curve $C$ in $M$.  At this point, we are willing to conjecture that the answer is Yes for acylindrical $3$-manifolds (though we do not have strong evidence to support this conjecture), but we are unwilling to advance this conjecture in the case where $M$ contains essential annuli.

\footnotesize{
 
}

\medskip
\noindent
{Mathematical Sciences\\University of Southampton\\Southampton SO17 1BJ\\England\\j.w.anderson@southampton.ac.uk}

\end{document}